\documentclass[english, 11pt, a4paper]{article}
\pdfoutput=1 
\RequirePackage[T1]{fontenc}
\RequirePackage[utf8]{inputenc}

\usepackage[
top=3cm,
bottom=3cm,
left=3cm,
right=3cm,
]{geometry}
\usepackage{setspace} 
\usepackage[nottoc, notlof, notlot]{tocbibind}
\usepackage{amsmath,amsfonts}
\usepackage{amsthm,amssymb}
\usepackage{letltxmacro}
\LetLtxMacro\amsproof\proof
\LetLtxMacro\amsendproof\endproof
\usepackage{thmtools}
\usepackage[overload]{empheq}
\allowdisplaybreaks
\usepackage[
    unicode=true,
    pdfstartview=FitV,
    colorlinks=true,
    citecolor=Leturquoise,
    linkcolor=Leturquoise,
    urlcolor=MFCB,
    linktoc=page,
    hyperindex=true,
    pdfcreator={},
]{hyperref}
\usepackage{cleveref}
\crefname{figure}{Figure}{Figures}
\Crefname{figure}{Figure}{Figures}
\crefname{section}{Section}{Sections}
\crefname{equation}{Equation}{Equations}
\crefname{remark}{Remark}{Remarks}
\crefname{Claim}{Claim}{Claims}
\crefname{Def}{Definition}{Definitions}
\crefname{Th}{Theorem}{Theorems}
\crefname{Cor}{Corollary}{Corollaries}
\crefname{Prop}{Proposition}{Propositions}
\crefname{Lmm}{Lemma}{Lemmas}
\crefname{Ex}{Example}{Examples}
\crefname{Exs}{Examples}{Examples}
\crefname{CEx}{Counter-example}{Counter-example}
\crefname{Q}{Question}{Questions}

\usepackage{bm}
\usepackage[eulergreek]{mathastext}
\usepackage{soul}
\sodef\allcapsspacing{\upshape}{0.15em}{0.65em}{0.6em}%
\sodef\lowsmallcapsspacing{\scshape}{0.075em}{0.5em}{0.6em}%

\newcommand{\spacedlowsmallcaps}[1]{\MakeLowercase{\textsc{\lowsmallcapsspacing{#1}}}}
\usepackage[english]{babel}
\usepackage{adforn}
\usepackage{fancyhdr}
\usepackage[titles]{tocloft}
\usepackage{titlesec}
\titleformat{\section}
{\relax}{%
  \makebox[0pt][r]{\textbf{\Large \MakeLowercase{\scshape \thesection}}\hspace*{10pt}}%
}{1em}{\Large\bfseries \spacedlowsmallcaps}
\titleformat{\subsection}
{\relax}{\textbf{\large \thesubsection}}{1em}{\large \bfseries}
\titleformat{\subsubsection}
{\relax}{{{\bfseries \thesubsubsection}}}{1em}{\bfseries \normalsize}        
\titleformat{\paragraph}[runin]
{\normalfont\normalsize}{\theparagraph}{0pt}{\spacedlowsmallcaps}

\usepackage{enumitem}
\setlist{nosep}
\titlespacing*{\chapter}{0pt}{1\baselineskip}{1.2\baselineskip}
\titlespacing*{\section}{0pt}{1.25\baselineskip}{1\baselineskip} 
\titlespacing*{\subsection}{0pt}{1.25\baselineskip}{1\baselineskip}
\titlespacing*{\paragraph}{0pt}{1\baselineskip}{1\baselineskip}
\usepackage[x11names,table]{xcolor} 
\usepackage{booktabs}
\definecolor{MCB}{cmyk}{0,0.03,0.08,0.26} 
\definecolor{MFCB}{cmyk}{0,0.06,0.20,0.6} 
\colorlet{Leturquoise}{DeepSkyBlue4}
\colorlet{TurquoiseClair}{DeepSkyBlue4!20}
\colorlet{MyOrange}{DarkOrange3!85}
\usepackage{graphicx, subcaption} 
\usepackage{tikz} 
\usetikzlibrary{calc,math,patterns}
\usetikzlibrary{arrows.meta} 
\usepackage{tcolorbox}
\newtcolorbox{CtBox}{colframe=DeepSkyBlue4!10.3,
  colback=DeepSkyBlue4!10.3,boxrule=0.0mm, halign=left}
\newtcolorbox{HypBox}{colframe=black,before=\vspace{0.3cm}\noindent, colback=black!0.3,boxrule=0.2mm, halign=left}
\newtcolorbox{IP}{colframe=black,colback=white, arc=0mm, bottomrule=0pt, toprule=0pt, rightrule=0pt, leftrule=5pt, halign=left}
\newcommand{\deffont}[1]{\emph{#1}}
\AtBeginDocument{%
  \LetLtxMacro\proof\amsproof
  \LetLtxMacro\endproof\amsendproof
}
\declaretheorem[
style=definition,
thmbox={style=S,underline=false,bodystyle=\normalfont \noindent,thickness=0.3pt},
name=Definition,
numberwithin=section,
refname={Definition,Definitions},
Refname={Definition,Definitions}]{Def}

\declaretheorem[
style=plain,
thmbox={style=S, underline=false, bodystyle=\normalfont \noindent},
name=Theorem,
sibling=Def,
refname={Theorem,Theorems},
Refname={Theorm,Theorems}]{Th}

\declaretheorem[
style=plain,
thmbox={style=S,underline=false,bodystyle=\normalfont \noindent},
name=Corollary,
sibling=Def,
refname={Corollary,Corollaries},
Refname={Corollary,Corollaries}]{Cor}

\declaretheorem[
style=plain,
thmbox={style=S,underline=false, bodystyle=\normalfont \noindent},
name=Proposition,
sibling=Def,
refname={Proposition,Propositions},
Refname={Propositions,Propositions}]{Prop}

\declaretheorem[
style=plain,
thmbox={style=S,underline=false,bodystyle=\normalfont \noindent},
name=Lemma,
sibling=Def,
heading={Lemma},
refname={Lemma,Lemmas},
Refname={Lemma,Lemmas}]{Lmm}

\theoremstyle{definition}
\newtheorem{Ex}[Def]{Example}

\newtheorem{Exs}[Def]{Examples}
\newtheorem{Q}[Def]{Question}
\newtheorem{Claim}[Def]{Claim}
\newtheorem{Rq}[Def]{Remark}

\numberwithin{equation}{section}
\newcommand{\bN}{\mathbb{N}}
\newcommand{\bZ}{\mathbb{Z}}

\newcommand{\bR}{\mathbb{R}}

\newcommand{\calC}{\mathcal{C}}

\newcommand{\diam}[1]{\mathrm{diam}\left( {#1} \right)}

\newcommand{\mbff}{\boldsymbol{f}}
\newcommand{\mbfg}{\boldsymbol{g}}
\newcommand{\mbfh}{\boldsymbol{h}}
\newcommand{\neutre}{\boldsymbol{e}}
\newcommand{\Gs}{\Gamma_m}
\newcommand{\Gsp}{{\Gamma^\prime}_m}
\newcommand{\Ds}{\Delta_m}
\newcommand{\WDs}{\Gs\wr\bZ} 
\newcommand{\BZ}{\Delta}
\newcommand{\TZ}{T^\prime}
\newcommand{\SZ}{\Sigma^\prime}
\newcommand{\RZ}{R^\prime}
\newcommand{\epsZ}{\varepsilon^\prime}
\newcommand{\range}{\mathrm{range}}

\newcommand{\supp}{\mathrm{supp}}

\newcommand{\profile}{I}
\newcommand{\landing}{\mathfrak{l}}
\newcommand{\Landing}{\mathfrak{L}}
\newcommand{\rhoaff}{\bar{\rho}}
\newcommand{\Cdiam}{C_R}
\newcommand{\Clm}{c}
\newcommand{\CTn}{C_{3}}
\usepackage{titling}
\title{\texorpdfstring{\textsc{Building prescribed quantitative orbit
      equivalence with $\bZ$}}{Building prescribed quantitative orbit
    equivalence with the group of integers}}
\author{\texorpdfstring{Amandine Escalier\thanks{Funded by Université
      Paris Cité and Sorbonne Université,
      CNRS, IMJ-PRG, F-75013 Paris, France and by the Deutsche
      Forschungsgemeinschaft (DFG, German Research Foundation) – Project-ID
      427320536 – SFB 1442, as well as under Germany’s 
      Excellence Strategy EXC 2044 –390685587, Mathematics Münster:
      Dynamics–Geometry–Structure.}}{Amandine Escalier}}
\date{\today}
\hypersetup{ 
pdfauthor={Amandine Escalier},
pdftitle={Building prescribed orbit equivalence with the group of integers},
pdfsubject={Orbit equivalence},
pdfkeywords={Følner, lamplighter, diagonal product, orbit, equivalence,
  isoperimetric profile, groups}
}

\begin{document}

\maketitle

\begin{abstract}
  Two groups are orbit equivalent if they both admit an action
  on a same probability space that share the same orbits. In particular the Ornstein-Weiss theorem implies that all infinite
  amenable groups are orbit equivalent to the group of integers. To refine this
  notion between infinite amenable groups Delabie,
  Koivisto, Le Maître and Tessera introduced a quantitative version of orbit
  equivalence. They furthermore obtained obstructions to 
  the existence of such equivalence using the isoperimetric profile.\\
  In this article we offer to answer the inverse problem (find a group being
  orbit equivalent to a prescribed group with prescribed
  quantification) in the case of the group of integers using the so called
  Følner tiling shifts introduced by Delabie et al. To do so we use the
  diagonal products defined by Brieussel and Zheng 
  giving groups with prescribed isoperimetric profile.
\end{abstract}
\vspace{0.5cm}
\noindent\textbf{Classification} 37A20
\vspace{0.75cm}

\section{Introduction} \label{Sec:Intro}
Two groups are orbit equivalent if they admit free measure-preserving actions on a same
standard probability space $(X, \mu)$ which share the same orbits. This notion
—emerging from the seminal work of Dye \cite{DyeI,DyeII}— can be seen as the
\emph{ergodic} version of the famous \emph{measure} equivalence introduced by
Gromov \cite{Gro}. A famous result of 
Ornstein and Weiss (see \cref{Th:OW}) implies that all amenable groups are
orbit equivalent. In particular —unlike quasi-isometry— orbit equivalence
does \emph{not} preserve coarse geometric invariants.

To overcome this issue it is therefore natural to look for some refinements of this 
orbit equivalence notion.
Assume for example that $G$ and $H$ are two finitely
generated orbit equivalent groups over a probability space $(X,\mu)$. Recall
that we can consider the Schreier graph associated to the action of $G$ (resp. $H$)
on $X$ and equip it with the usual metric $d_{S_G}$ (resp. $d_{S_H}$), fixing the
length of an edge to one. 
A first way to refine the measure equivalence is to quantify how close the two
actions are by studying for all $g\in G$ and $h\in H$ the integrability of the
two following maps 
\begin{equation*}
  x\mapsto d_{S_G}(x,h\cdot x) \qquad x\mapsto d_{S_H}(x,g\cdot x). 
\end{equation*}
When these two maps are $L^p$ we say that the groups are $L^p$-orbit
equivalent (see \cite{BFSIntegrability} for more details). In this refined
framework a famous result of Bader, Furman and Sauer \cite{BFSIntegrability}
implies that
any group $L^1$-orbit equivalent to a lattice in $SO(n, 1)$ for some $n\geq 2$ is 
virtually a lattice in $SO(n,1)$. This refinement also lead Bowen to prove in
the appendix of \cite{AustinBowen} that volume growth was invariant under
$L^1$-orbit equivalence.

Delabie, Koivisto, Le Maître and Tessera offered in \cite{DKLMT} to
extend this quantification to a family of functions larger than $\{x\mapsto x^p,
\ p\in [0,+\infty]\}$ (see \cref{Def:QuantMev}).
They furthermore showed the monotonicity of the isoperimetric profile under this  
quantified measure equivalence definition (see \cref{Th:ProfiletOE}). In
\cite{BZ} Brieussel and Zheng managed to construct amenable groups with
prescribed isoperimetric profile called \emph{diagonal product}. Considering the
monotonicity of the isoperimetric profile, the striking result of Brieussel and
Zheng thus triggers a new  question: instead of trying to quantify the
equivalence relation between two given groups, can one find a group that is
orbit equivalent to a prescribed group with a prescribed quantification? 

This is the problem we address in this article. Using Brieussel-Zheng’s construction we
exhibit a group that is orbit equivalent to $\bZ$ with a prescribed
quantification (see \cref{Th:CouplingwithZ}). Comparing the obtained coupling
to the constraints given by \cref{Th:ProfiletOE} 
we show that our coupling is close to being optimal for a sense of “optimal”
that we make precise in \cref{Subsec:IP}.

\subsection{Quantitative orbit equivalence}\label{Subsec:ME}
Let us recall some material from \cite{DKLMT}. A \emph{measure-preserving
  action}\index{Measure!Preserving action} 
of a discrete countable group $G$ on a measured space $(X,\mu)$ is an action of
$G$ on $X$ such that the map $(g,x)\mapsto g\cdot x$ is a Borel map and
$\mu(E)=\mu(g\cdot E)$ for all $E\subseteq \mathcal{B}(X)$ and all $g\in G$. We
will say that a measure-preserving action of $G$ on $(X,\mu)$ is
\emph{free}\index{Free action} if for almost every $x\in X$ we have $g\cdot x =
x$ if and only if $g=\neutre_G$.

We recall below the definition of orbit equivalence and the
quantified version as introduced by Delabie, Koivisto, Le Maître and Tessera
\cite{DKLMT}. We conclude this section by studying the relation between isoperimetric profile
and orbit equivalence. 

\begin{Def} Let $G$ and $H$ be two finitely generated groups. We say that $G$
  and $H$ are \deffont{orbit equivalent} if there exists a probability space
  $(X,\mu)$ and a measure-preserving free action of $G$ (resp. $H$) on $(X,\mu)$
  such that for almost every $x \in X$ we have $G\cdot x = H\cdot
  x$.\index{Orbit!Equivalent}
  We call $(X,\mu)$ an \deffont{orbit equivalence coupling} from $G$ to $H$.
\end{Def}

By the Ornstein Weiss theorem \cite[Th. 6]{OW80} below, all infinite amenable
groups are in the same equivalence class.  
\begin{Th}[{\cite{OW80}}]\label{Th:OW}
  All infinite amenable groups are orbit equivalent to $\bZ$. 
\end{Th}

To refine this equivalence relation and “distinguish” amenable groups we
introduce the quantified version of orbit equivalence.

Recall that if a finitely generated group $G$ acts on a space $X$ and if $S_G$
is a finite generating set of $G$, we can define the Schreier graph associated to this
action as being the graph whose set of vertices is $X$ and set of edges is
$\{(x,s\cdot x) \, | \, s\in S_K\}$. This graph is endowed with a natural metric
$d_{S_G}$ fixing the length of an edge to one. Remark that if $S^\prime_G$ is another
generating set of $G$ then there exists $C>0$ such that for all $x\in X$ and
$g\in G$
\begin{equation*}
 \frac{1}{C}d_{S_G}(x,g\cdot x) \leq d_{S^{\prime}_G}(x,g\cdot x) \leq Cd_{S_G}(x,g\cdot x).
\end{equation*}

\begin{Def}[{\cite[Def. 2.18]{DKLMT}}]\label{Def:QuantMev}
  \index{Phi-Psi-integrability@$(\varphi,\psi)$-integrability}
  We say that an orbit equivalence coupling
  $(X,\mu)$ from $G$ to $H$ is \deffont{$(\varphi,\psi)$-integrable} if
  for all $g\in G$ (resp. $h\in H$) there exists $c_g>0$ (resp. $c_h>0$) such
  that
  \begin{equation*}
    \int_{X} \varphi\left( \frac{1}{c_g} d_{S_H}(g\cdot x, x)
    \right)\mathrm{d}\mu(x)<+\infty \quad \text{and} \quad
    \int_{X}\psi\left( \frac{1}{c_h} d_{S_G}(h\cdot x, x)\right)
    \mathrm{d}\mu(x)<+\infty.
  \end{equation*} 
\end{Def}
We introduce the constants $c_g$ and $c_h$ in the definition for the 
integrability to be independent of the choice of generating sets $S_G$ and
$S_H$. If $\varphi(x)=x^p$ we will sometimes talk of $(L^p,\psi)$-integrability
instead of $(\varphi,\psi)$-integrability. In particular $L^0$ means
that no integrability assumption is made.
Finally, note that every $(L^\infty,\psi)$-integrable coupling is
$(\varphi,\psi)$-integrable for any increasing map $\varphi: \bR^+\rightarrow
\bR^+$.  When $\varphi=\psi$ we will say that the coupling is
$\varphi$-integrable instead of $(\varphi, \varphi)$-integrable.  

\begin{Exs}[\cite{DKLMT}]\label{Ex:CouplagesQt}\leavevmode 
  \begin{enumerate}
  \item There exists an orbit
    equivalence coupling between $\bZ^4$ and the Heisenberg 
    group $\mathrm{Heis}(\bZ)$ that is $L^p$-integrable for all $p<1$. 
  \item Let $k\in \bN^{*}$. Their exists an $(L^{\infty},\exp)$-integrable orbit
    equivalence coupling from the lamplighter group to the Baumslag-Solitar
    group $BS(1,k)$.
  \end{enumerate}
\end{Exs}

More examples will be given in \cref{Subsec:FOTS}. Let us conclude on the
quantification by a remark. We chose to refine orbit
equivalence using the \emph{integrable} point of view. But it is not the only
possible sharpening. For example Kerr and Li \cite{KerrLi} defined
\emph{Shannon orbit equivalence}: instead of looking at the integrability of
distance maps they consider the Shannon \emph{entropy} of partitions associated
to the coupling.

\subsection{Isoperimetric profile}\label{Subsec:IP}
As stated before, the orbit equivalence does not preserve the coarse geometric 
invariants. But the quantified version defined above allowed Delabie
et al. \cite{DKLMT} to get a relation between the isoperimetric profiles of two
orbit equivalent groups which we describe below.

Recall that if $G$ is generated by a finite set $S$, the \deffont{isoperimetric
  profile} of $G$ is defined as\footnote{We chose to adopt the convention of
  \cite{DKLMT}. Note that in \cite{BZ}, the isoperimetric profile is defined as
  $\Lambda_G=1/\profile_G$.} 
\begin{equation*}
  \profile_{G}(n):= \sup_{|A|\leq n} \frac{|A|}{|\partial A|}.
\end{equation*}
For example the isoperimetric profile of $\bZ$ verifies $\profile_{\bZ}(x) \simeq x$. 
Remark that due to Følner criterion, a group is amenable if and only if its
isoperimetric profile is unbounded. Hence we can see the isoperimetric profile
as a way to measure the amenability of a group: the faster $\profile_G$ tends to
infinity, the more amenable $G$ is.

The behaviour of the isoperimetric profile under measure equivalence
coupling is given by the theorem below. If $f$ and $g$ are two real functions we
denote $f \preccurlyeq g$ if there exists some constant $C>0$ such that
$f(x)=\mathcal{O}\big(g(Cx)\big)$ as $x$ tends to infinity. We write $f\simeq g$
if $f \preccurlyeq g$ and $g\preccurlyeq f$.

\begin{Th}[{\cite[Th.1]{DKLMT}}] \label{Th:ProfiletOE}
  Let $G$ and $H$ be two finitely generated groups admitting a
  $(\varphi,L^0)$-integrable orbit equivalence coupling. If $\varphi$ and
  $t/\varphi(t)$ are non-decreasing then
  \begin{equation*}
    \varphi \circ \profile_{H} \preccurlyeq \profile_G.
  \end{equation*}
\end{Th}

This theorem provides an obstruction for finding $\varphi$-integrable couplings
with certain functions $\varphi$ between two amenable groups. For example
for a coupling with $H=\bZ$ the integrability has to verify $\varphi \preccurlyeq
I_G$.
This lead the authors of \cite{DKLMT} to ask the following question.
\begin{Q}[{\cite[Question 1.2]{DKLMT}}]\label{Q:DKLMTZ}
  Given an amenable finitely generated group $G$, does there exist a
  $(I_G,L^0)$-integrable orbit equivalence coupling from $G$ to $\bZ$?
\end{Q}

We answer the above question for a large family of maps
$\varphi$ in \cref{Th:CouplingwithZ}. We will see that the coupling
we build to proof the aforementioned theorem answers \cref{Q:DKLMTZ}
up to a logarithmic error.

\subsection{Main results}
In this paper we show the following main theorem and its corollary below.

\begin{restatable}{Th}{CouplingwithZ}\label{Th:CouplingwithZ}
  For all non-decreasing function $\rho: [1,+\infty[
  \rightarrow [1,+\infty[$ such that $\rho(1)=1$ and $x/\rho(x)$ is
  non-decreasing, there exists a group $G$ such that
  \begin{itemize}
  \item $I_{G}\simeq \rho \circ \log$;
  \item there exists an orbit equivalence coupling from $G$ to $\bZ$ that is
    $(\varphi_{\varepsilon},\exp \circ \rho)$-integrable for all
    $\varepsilon>0$, where $\varphi_{\varepsilon}(x):={\rho\circ
      \log(x)}/{\left( \log\circ\rho\circ\log(x) \right)^{1+\varepsilon}}$.
  \end{itemize}
\end{restatable}

Let us discuss the optimality of this result. Consider a
$(\varphi,L^0)$-integrable orbit equivalence coupling from some group $G$ to
$\bZ$. By \cref{Th:ProfiletOE} it verifies $\varphi\circ I_{\bZ} \preccurlyeq
I_G$. In particular since $I_{\bZ}(x)\simeq x$, we can not have a better
integrability than $\varphi(x)\simeq I_G$. Since $I_{G}\simeq \rho \circ
\log$ our above theorem is optimal up to a logarithmic error. We discuss this in
more length in \cref{Sec:Conclusion}.

\paragraph{Main Ingredients} The main tools of the proof of
\cref{Th:CouplingwithZ} are Brieussel-Zheng’s 
diagonal products (see \cref{Sec:BZGroupsZ}) and Følner tiling shifts (see
\cref{Sec:FOTS}). We show that a diagonal product $\BZ$ admits a coupling with
$\bZ$ satisfying \cref{Th:CouplingwithZ}. To prove it we use the integrability
criterion given by \cref{Prop:ConditionpourOE} and involving Følner tiling
shifts. 

Therefore we compute in \cref{subsec:FOTSDelta} a Følner tiling shift
$(\Sigma_n)_n$ for $\BZ$. We also estimate the tiles’ diameter and the
proportion of elements in the boundary. We construct a Følner tiling shift for
$\bZ$ in \cref{subsec:TilesZ} and show that these two tiling shifts verify
\cref{Prop:ConditionpourOE}.

\bigskip

Let us now consider the possible generalisations of this result to other
groups than the group of integers. 
To do so we can use the \emph{composition} of couplings described in
\cite[Section~2]{DKLMT}.

Given the above theorem, once we have a measure equivalence coupling from $\bZ$
to a group $H$ we can compose the two couplings to obtain a measure equivalence
from $G$ to $H$.
If the growth of the isoperimetric profile of $H$ is close to the one of $\bZ$, the
integrability of the obtained coupling will be close to the optimal one given by
\cref{Th:ProfiletOE}. It is for example the case when $H=\bZ^d$.

\begin{Cor}\label{Rq:Zd} Let $d\in \bN^{*}$ and $\varepsilon>0$.
  For all non-decreasing function $\rho: [1,+\infty[
  \rightarrow [1,+\infty[$ such that $\rho(1)=1$ and $x/\rho(x)$ is
  non-decreasing,
  if the map $\varphi_\varepsilon$ defined in \cref{Th:CouplingwithZ}
  is subadditive and concave, then there exists a group $G$ such that
  \begin{itemize}
  \item $I_{G}\simeq \rho \circ \log$ ;
  \item there exists a $(\varphi_{\varepsilon},L^0)$-integrable orbit
  equivalence coupling from $G$ to $\bZ^d$.
  \end{itemize}
\end{Cor}

\paragraph{Structure of the paper} In \cref{Sec:BZGroupsZ} we present the
diagonal products introduced by Brieussel and Zheng. We recall some of the
properties shown in \cite{BZ} and compute Følner sequences.
\Cref{Sec:FOTS} is devoted to Følner tiling shifts. These tools built by Delabie
et al. \cite{DKLMT} allow us to construct and quantify an orbit equivalence coupling
between two groups. In this section we also construct Følner tiling shifts for
diagonal products $\BZ$.
We show our main theorem in \cref{Sec:CouplageZ} combining the results of the
two previous sections. Finally we discuss the limits of this construction and
some open problems in \cref{Sec:Conclusion}.

\begin{center}
  \adforn{24}
\end{center}

\paragraph{Acknowledgements} 
I would like to thank Romain Tessera and Jérémie Brieussel, under
whose supervision the work presented in this article was carried out.
I thank them for suggesting the topic, sharing their precious insights and for
their many useful advice. I also thank the anonymous referee for their remarks
and corrections.
\section{Diagonal products of lamplighter groups}\label{Sec:BZGroupsZ}
\label{Sec:DefdeBZ}
We recall here necessary material from \cite{BZ} concerning the definition of
\emph{Brieussel-Zheng’s diagonal products}. We give the definition of such
a group, recall and prove some results concerning the range (see
\cref{Def:Range}) of an element and use it to identify a Følner
sequence. Finally we present in \cref{Subsec:FromIPtoBZ} the tools needed to
recover such a diagonal product starting with a prescribed isoperimetric
profile.

\subsection{Definition of diagonal products} \label{subsec:DefDelta}
Recall that the wreath product of a group $G$ with $\bZ$ denoted $G \wr \bZ$ is
defined as $G\wr \bZ:= \oplus_{m \in \bZ} G \rtimes \bZ$. An element of $G\wr
\bZ$ is a pair $(f,t)$ where $f$ is a map from $\bZ$ to $G$ with finite support
and $t$ belongs to $\bZ$. We refer to $f$ as the \deffont{lamp configuration}
and $t$ as the \deffont{cursor.} Finally we denote by $\supp(f)$ the
\deffont{support} of $f$ which is defined as
$\supp(f):=\left\{x\in \bZ \ | \ f(x)\neq e_G \right\}$.

\subsubsection{General definition}
Let $A$ and $B$ be two finite groups. Let $(\Gamma_m)_{m\in \bN}$ be a sequence
of finite groups such that each $\Gamma_m$ admits a generating set of the form $A_m \cup B_m$
where $A_m$ and $B_m$ are finite subgroups of $\Gamma_m$ isomorphic respectively
to $A$ and $B$. For $a \in A$ we denote $a_m$ the copy of $a$ in $A_m$ and
similarly for $B_m$.

Finally let $(k_m)_{m \in \bN}$ be a sequence of integers such that $k_{m+1} \geq 2 k_m$
for all $m$. We define $\Ds=\WDs$ and endow it with the generating set  \label{Def:SDeltam}
\begin{equation*}
  S_{\Ds}:=
  \Big\{ (\mathrm{id},1 )\Big\} \cup \Big\{ \big(a_m \delta_0,0 \big) \ | \ a_m \in A_m\Big\}
  \cup \Big\{ \big(b_m \delta_{k_m},0 \big) \ | \ b_m \in B_m\Big\}.
\end{equation*}

\begin{Def}\label{Def:BZGr}\index{Diagonal product} \index{Brieussel-Zheng’s
    diagonal product}
  The \deffont{Brieussel-Zheng diagonal product} associated to
  $(\Gamma_m)_{m\in \bN}$ and $(k_m)_{m\in \bN}$ is the subgroup $\BZ$ of $ \left(\prod_m 
    \Gamma_m\right)\wr \bZ$ generated by 
  \begin{equation*}
    S_{\BZ}:=
    \Big\{ \Big({ {\big(\mathrm{id}\big)}_m,1 }\Big)\Big\} \cup
    \Big\{ \big( {\left(a_m \delta_0  \right)}_m,0 \big) \ | \ a \in A\Big\} 
    \cup \Big\{ \big( {\left( b_m \delta_{k_m} \right)}_m,0 \big) \ | \ b \in B\Big\}.
  \end{equation*}
\end{Def}

The group $\BZ$ is uniquely determined by the sequences $(\Gs)_{m\in \bN}$ and
$(k_m)_{m\in \bN}$. Let us give an illustration of what an element in such a
group looks like. We will denote by $\mbfg$ the sequence $(g_m)_{m\in \bN}$.

\begin{Ex} We
  represent in \cref{fig:ab} the element $\big(\mbfg,t \big)$ of $\BZ$ verifying
  \begin{equation*}
    (\mbfg,t)=\big( (g_m)_{m\in \bN},t \big):=\big({\left(a_m\delta_0\right)}_m,0 \big)
    \big( {\left( b_m \delta_{k_m}\right)}_m,0\big) (0,3),
  \end{equation*}
  when $k_m=2^m$. The cursor is represented by the blue arrow at the bottom of the
  figure. The only value of $g_0$ different from the identity is $g_0(0)=(a_0,b_0)$.
  Now if $m>0$ then the only values of $g_m$ different from the identity are
  $g_m(0)=a_m$ and $g_m(k_m)=b_m$.  
\end{Ex}

\begin{figure}[htbp]
  \centering
  \includegraphics[width=0.8\textwidth]{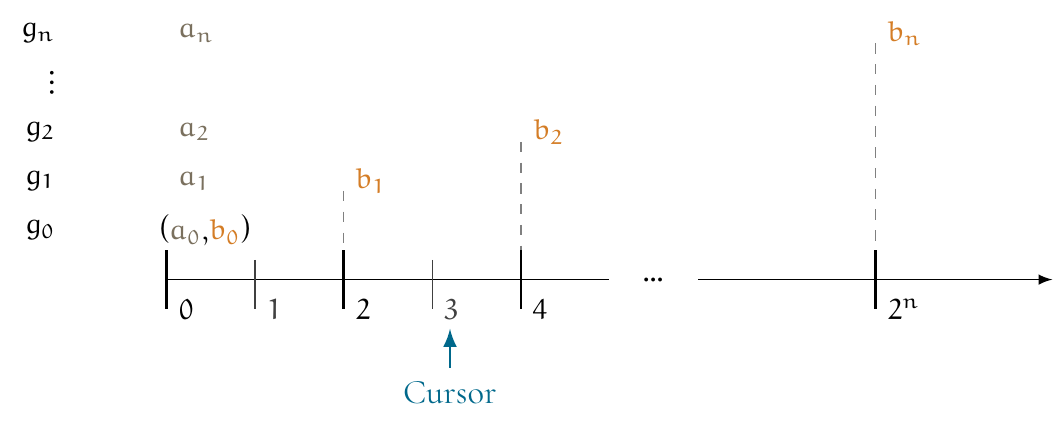}
  \caption{Representation of $\big(\mbfg,t\big)=\big((a_m\delta_0)_m,0\big)%
    \big((b_m\delta_{k_m})_m,0\big)(0,3)$ when $k_m=2^m$.}
  \label{fig:ab}
\end{figure}

\subsubsection{The expanders case}\label{subsec:Expanders}
In this article we will restrict ourselves to a particular familiy of groups
$(\Gamma_m)_{m\in \bN}$ called \deffont{expanders.}  Recall that $(\Gamma_m)_{m\in \bN}$
is said to be a sequence of \deffont{expanders}\index{Expander} if the sequence
of diameters $(\diam{\Gamma_m})_{m\in \bN}$ is unbounded and if there exists
$c_0>0$ such that for all $m\in \bN$ and all $n\leq |\Gamma_m|/2$ the
isoperimetric profile verifies $I_{\Gamma_m}(n)\leq c_0$.

When talking about diagonal products we will always make the following
assumptions. We refer to \cite[Example 2.3]{BZ} for an explicit example of
diagonal product verifying \textbf{(H)}.

\begin{HypBox}
  \textbf{Hypothesis (H)}
  \begin{itemize}
  \item $(k_m)_m$ and $(l_m)_m$ are sub-sequences of geometric
    sequences;
  \item $k_{m+1}\geq 2 k_m$ for all $m\in \bN$;
  \item $(\Gamma_m)_{m\in \bN}$ is a sequence of expanders such that $\Gamma_m$
    is a quotient of $A*B$ and there exists $\Clm >0$ such that
    $1/\Clm l_m \leq \diam{\Gamma_m}\leq \Clm l_m$ for all $m\in \bN$;
  \item $k_0=0$ and $\Gamma_0=A_0\times B_0$;
  \item the natural quotient map $A_m\times B_m\rightarrow
      \langle{ \langle{\left[A_m,B_m\right]}}  \rangle \rangle \backslash \Gamma_m$
    is an isomorphism,  where $\langle{\langle{\left[A_m,B_m\right]}\rangle}\rangle $
    is the normal
    closure of $\left[A_m,B_m\right]$.
  \end{itemize}
\end{HypBox}
Recall (see \cite[page 9]{BZ}) that in this case there exist $c_1$, $c_2>0$ such
that, for all $m$
\begin{equation}
  \label{eq:encadrementlngammap}
  c_1 l_m -c_2 \leq \ln \left\vert {\Gamma_m}\right\vert \leq c_1l_m + c_2.
\end{equation}

Finally we adopt the convention of \cite[Notation 2.2]{BZ} and allow $(k_m)_{m\in \bN}$
to take the value $+\infty$. In this case $\BZ_s$ is the trivial group. In
particular when $k_1=+\infty$ the diagonal product $\BZ$ corresponds to the
usual lamplighter $(A\times B)\wr \bZ$. 

\subsubsection{Relative commutators subgroups}\label{subsec:Derivedfunctions}

Let $\theta_m : \Gs \rightarrow 
\langle{ \langle{\left[A_m,B_m\right]}}  \rangle \rangle \backslash \Gs \simeq
A_m \times B_m$ be the natural projection, for all $m\in \bN$. Let $\theta^A_m$ and $\theta^B_m$
denote the composition of $\theta_m$ with the projection to $A_m$ and $B_m$
respectively. Now let $m\in \bN$ and define $\Gsp:= \langle{ \langle{
    \left[A_m,B_m\right]}} \rangle \rangle$. If $(g_m,t)$ belongs to $\Delta_m$
then there exists a unique $g^{\prime}_m \ : \ \bZ \rightarrow \Gsp$ such that
$g_m(x)=g^{\prime}_m(x)\theta_m\big(g_m(x)\big)$ for all $x\in \bZ$.
\begin{Ex}
  Let $(\mbfg,3)$ be the element described in \cref{fig:ab}. Then
  the only non-trivial value of $\theta_0(g_0)$ is 
  $\theta_0(g_0(0))=(a_0,b_0)$. If $m>0$ then the only non trivial values of
  $\theta_m(g_m)$ are $\theta_m(g_m(0))=(a_m,\neutre)$ and
  $\theta_m(g_m(k_m))=(\neutre,b_m)$. Finally for all $m$ we have 
  $g^{\prime}_m=\mathrm{id}$ since there are no commutators appearing in the
  decomposition of $(\mbfg,0)$.
\end{Ex}
\begin{Ex}
  \label{Ex:Commutateur}
  Assume that $k_m=2^m$ and consider first the element $(\mbff,0)$ of $\BZ$
  defined by $(\mbff,0):=(0,-k_1)\big((a_m\delta_0)_{m},0 \big)(0,k_1)$. Now
  define the commutator 
  \begin{equation*}
    (\mbfg,0)=(\mbff,0)\cdot \big((b_m\delta_{k_m})_{m},0\big)\cdot
    (\mbff,0)^{-1} \cdot \big((b^{-1}_m\delta_{k_m})_{m},0\big)
  \end{equation*}
  and let us describe the values taken by $\mbfg$ and the induced maps
  $\theta_m(g_m)$ and $g^{\prime}_m$ (see \cref{fig:Commutateur} for a
  representation of $\mbfg$). The only non-trivial commutator appearing in the
  values taken by $\mbfg$ is $g_1(k_1)$ which is equal to $
  a_1b_1a^{-1}_1b^{-1}_1$. In other words
  $g_0$ is the identity, thus 
  $\theta_0=\mathrm{id}$. Moreover when $m=1$ 
  we have $\theta_1=\mathrm{id}$ and the only value of $g^{\prime}_1(x)$
  different from $\neutre$ is $g^{\prime}_1(k_1)= a_1b_1a^{-1}_1b^{-1}_1$ (on a
  blue background in \cref{fig:Commutateur}). Finally
  if $m>1$ then $g_m$ is the identity thus $\theta_m=\mathrm{id}$ and
  $g^{\prime}_m=\mathrm{id}$.
\end{Ex}
\begin{figure}[htbp]
  \centering
  \includegraphics[width=0.90\textwidth]{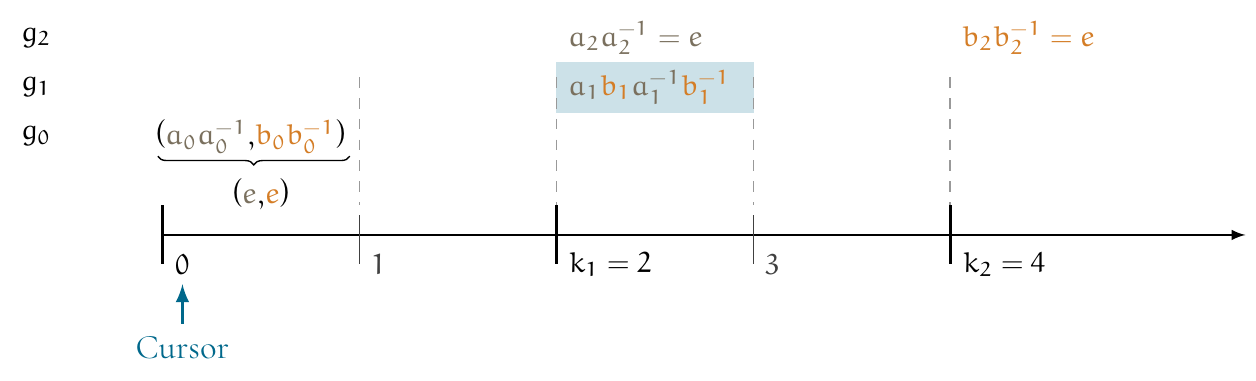}
  \caption[\cref{Ex:Commutateur}]{Representation of $(\mbfg,0)$ defined in
    \cref{Ex:Commutateur}}
  \label{fig:Commutateur}
\end{figure}
Let us study the behaviour of this decomposition under product of
lamp configurations.

\begin{Claim}\label{Claim:DeriveeProduit}
  If $g_m,f_m \ : \ \bZ \rightarrow \Gamma_m$ then
   $(g_mf_m)^{\prime}=g^{\prime}_m\theta_m(g_m)f^{\prime}_m\Big(\theta_m(g_m) \Big)^{-1}$.
\end{Claim}
\begin{proof}
  Since  $g_m=\theta_m(g_m) g^{\prime}_m$ and  $f_m=\theta_m(f_m) f^{\prime}_m$
  we can write
  \begin{align*}
    g_mf_m \
    = \ g^{\prime}_m\theta_m(g_m) \cdot f^{\prime}_m\theta_m(f_m)\ 
    = \ g^{\prime}_m\theta_m(g_m)f^{\prime}_m\theta_m(g_m)^{-1}\theta_m(g_m)\theta_m(f_m).
  \end{align*}
  But $\theta_m(g_m)\theta_m(f_m)$ takes values in $A_m\times B_m$ and
  $\Gsp$ is a normal subgroup of $\Gamma_m$ thus the map
  $g^{\prime}_m\theta_m(g_m)f^{\prime}_m\theta_m(g_m)^{-1}$ takes
  values in $\Gsp$. Hence the claim.
\end{proof}
Combining Lemma 2.7 and Fact 2.9 of
\cite{BZ}, we get the following result.

\begin{Lmm}
  \label{Lmm:deffm}
  Let $(\mbfg, t) \in \BZ$. For all $m\in \bN$
  and $x\in \bZ$
  \begin{align*}
    g_m(x)&=g^\prime_m(x) \theta^A_m\big(g_m(x)\big)\theta^B_m\big(g_m(x)\big)\\
    &=g^\prime_m(x) \theta^A_m\big(g_0(x)\big)\theta^B_m\big(g_0(x-k_m)\big).
  \end{align*}
  In particular the sequence $\mbfg=\left( g_m \right)_{m\in \bN}$ is uniquely determined
  by $g_0$ and $\left( g^\prime_m \right)_{m\in \bN}$. 
\end{Lmm}
In the next subsection we are going to see that we actually need only a
\emph{finite} number of elements of the sequence $(g^{\prime}_m)_{m \in \bN}$ to
characterise $\mbfg$.

\subsection{Range and support}
In this subsection we introduce the notion of \emph{range} of an element $(\mbfg,t)$ in $\BZ$
and link it to the supports of the lamp configurations $(g_m)_{m\in \bN}$.

\subsubsection{Range}
We denote by $\pi_{2} : \Delta \rightarrow \bZ$ the projection on the
second factor and for all $n\in \bN$ denote by $\landing(n)$ the integer such
that $k_{\landing(n)}\leq n< k_{\landing(n)+1} $.
\begin{Def}\label{Def:Range}
  If $w=s_1\ldots s_m$ is a word over $S_\BZ$ we define its
  \deffont{range}\index{Range!Of a word} as
  \begin{equation*}
    \range(w) := \left\{\pi_2\left(\prod^{i}_{j=1}s_j \right) \,
      | \, i=0,\ldots,n \right\}.
  \end{equation*}
\end{Def}
The range is a finite subinterval of $\bZ$. It represents the set of sites
visited by the cursor. 
\begin{Def}\label{Def:Rangedelta}\index{Range!Of an element}
  The \deffont{range} of an element $\delta \in \Delta$ is defined as the
  diameter of a minimal range interval of a word over $S_{\BZ}$ representing $\delta$.
\end{Def}
In what follows we will consider elements that can be written as a word with range
in an interval of the form $[0,n]$, where $n$ belongs to $\bN$. Therefore,
when there is no ambiguity we will denote $\range(\delta)$ this interval,
namely $\range(\delta)=[0,n]$.

\begin{Ex}
  Let $(\mbfg,0) \in \BZ$ such that
  $\range(\mbfg,0)=[0,6]$, that is to say: the cursor can only
  visit sites between $0$ and $6$. Then the map $g_m$ can “write” elements of
  $A_m$ only on sites visited by the cursor, that is to say from $0$ to $6$, and
  it can write elements of $B_m$ only from $k_m$ to $6+k_m$. Thus $g_0$ is
  supported on $[0,6]$, since $k_0=0$. Moreover, commutators (and hence elements
  of $\Gsp$) can only appear between $k_m$ and $6$, thus $\supp(g^\prime_m)
  \subseteq [k_m,6]$. In particular $\supp(g^\prime_m)$ is empty when $k_m>6$.
  
  Such a $(\mbfg,0)$ is represented in \cref{fig:elementdeDelta} for $k_m=2^m$.
\end{Ex}
\begin{figure}[htbp]
  \centering
  \includegraphics*[width=\textwidth]{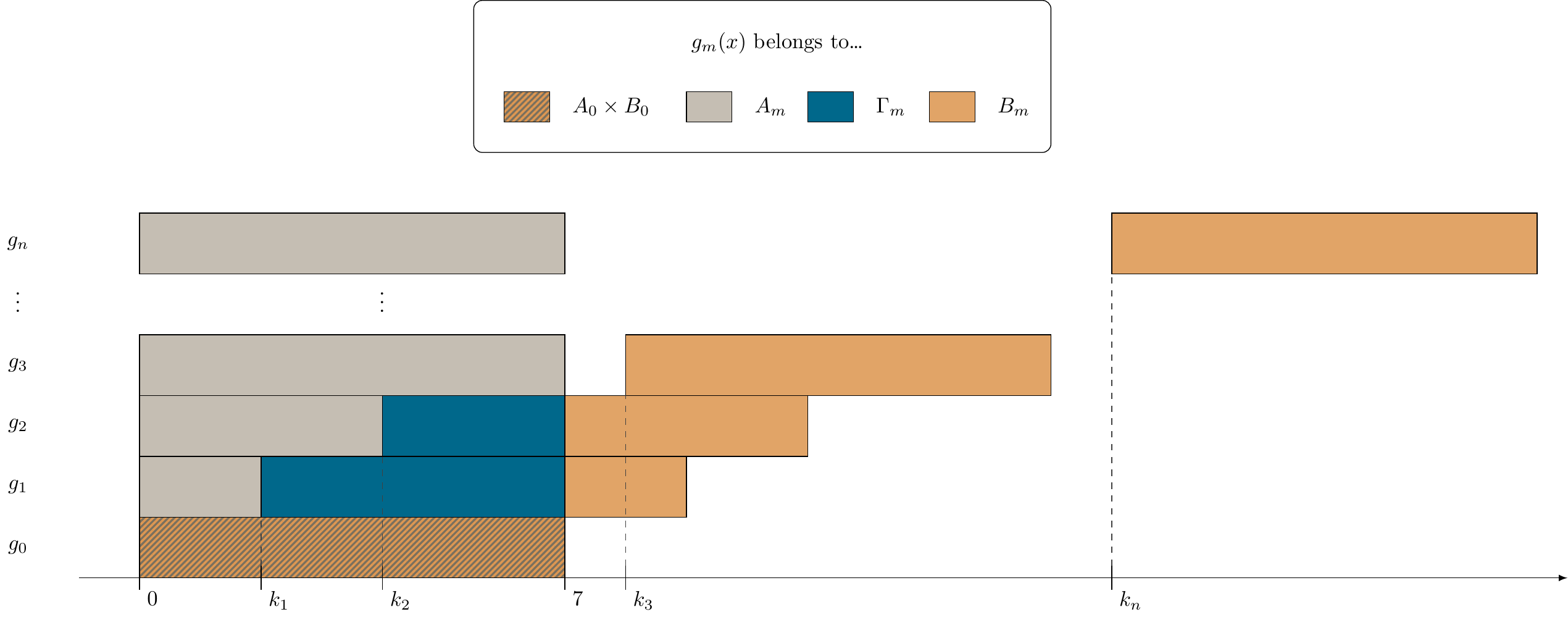} 
  \caption{An element of $\BZ$}
  \label{fig:elementdeDelta}
  \vspace{0.5cm}
  {\small 
    Recall that $g_m$ : $\bZ \rightarrow
    \Gamma_m$. If $m\leq \landing(6)$, then $g_m(x)$ belongs to $A_m$ if $x \in
    [0,k_m-1]$, it belongs to $\Gs$ if $x\in [k_m,6]$ and to $B_m$ if $x\in [7,6+k_m]$ and
    equals $\neutre$ elsewhere. If $m>\landing(6)$ then $g_m(x)$ belongs to $A_m$ if $x \in
    [0,6]$ and to $B_m$ if $x\in [k_m,6+k_m]$ and equals $\neutre$ elsewhere.
  }
\end{figure}

Let us now recall a useful fact proved in \cite{BZ}.
\begin{Claim}[{\cite[Fact 2.9]{BZ}}]
  An element $(\mbfg,t) \in \BZ$ is uniquely determined by $t$, $g_0$ and the sequence
  $ (g^{\prime}_m)_{m\leq \landing(\range(\mbfg,t))}$. 
\end{Claim}
\begin{Ex}
  Consider again $(\mbfg,0) \in \BZ$ such that
  $\range(\mbfg,0)=[0,6]$, which was illustrated in \cref{fig:elementdeDelta}. 
  Since $k_3=8>6$, the element $(\mbfg,0)$ is uniquely determined by the data
  $g_0$ (that is to say, the values read in the bottom line) and the values of
  $g^\prime_i$ for $i=1,2$ (namely, the value taken in the blue area).
  \Cref{fig:troncatureagprime} represents the aforementioned characterizing data. 
\end{Ex}
\begin{figure}[htbp]
  \centering
  \includegraphics*[width=0.55\textwidth]{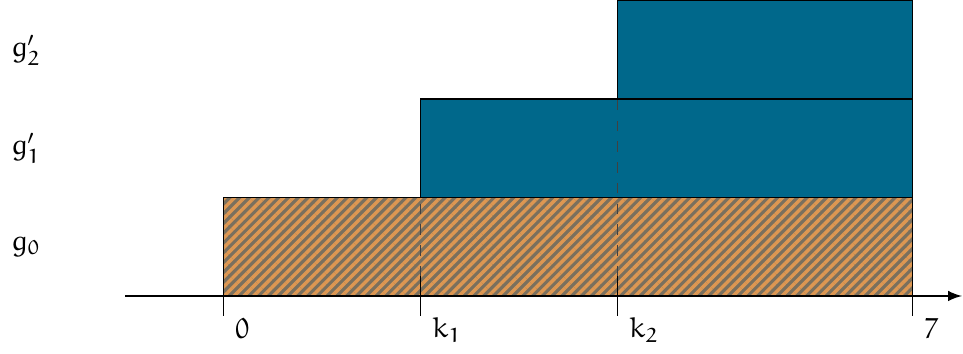} 
  \caption{Data needed to characterise $\mbfg$ such that $\range(\mbfg)\subset
    [0,6]$ when $k_m=2^m$}
  \label{fig:troncatureagprime}
\end{figure}
\subsubsection{Relation between range and support}
Recall that for all $m\in \bN$ we can write
$g_m(x)=g^\prime_m(x)\theta^A_m\big(g_0(x)\big)\theta^B_m\big(g_0(x-k_m)\big)$ and
that $\landing(n)$ denotes the integer such that $k_{\landing(n)}\leq n <
k_{\landing(n)+1}$.

To work with the Følner sequence we compute in \cref{Subsec:FolnerSequence} and
deduce a Følner \emph{tiling} shift from it, we will need to link the range
of $(\mbfg,t)$ in $\BZ$ with the support of $g_0$ and the sequence of supports of
$({g^\prime}_m)_{m \in \bN}$. This is what the following lemma formalises.

\begin{Lmm}\label{Lmm:range}
  Let $n\in \bN$ and take $(\mbfg,t) \in \Delta$.
  Then $\range(\mbfg,t)$ is included in $ [0,n] $ if and only if
  \begin{equation*}
    \begin{cases}
      t\in [0,n] \\
      \supp(g_0) \subset [0, n]\\
      \supp({g^\prime}_m) \subseteq [k_m,  n] \quad &\forall 1\leq m \leq \landing(n)  \\
      {g^\prime}_m\equiv \neutre \quad &\forall m > \landing(n).
    \end{cases}
  \end{equation*}
\end{Lmm}
\begin{proof}
  Let $n \in \bN$ and first assume that $\range(\mbfg,t)\subseteq [0,n]$, that
  is to say: the cursor 
  can only visit sites between $0$ and $n$. Let $(\mbfg,t)=\prod^{l}_{i=0} s_i$
  be a decomposition in a product of elements of $S_{\BZ}$ with range of minimal length.
  Let $m\in \bN$, then by definition of 
  $S_{\BZ}$, an element $s_i$ can “write” elements of $A_m$ only between $0$ and
  $n$, and it can write elements of $B_m$ only between $k_m$ and $n+k_m$. Thus $g_0$ is
  supported on $[0,n]$, since $k_0=0$. And commutators can only appear between
  $k_m$ and $n$, hence $\supp(g^\prime_m)\subseteq [k_m,n]$. In particular if
  $k_m>n$ then $g^\prime_m \equiv \neutre$. Finally we obtain that $t$ belongs
  to $[0,n]$ by noting that $t=\pi_2\left(\prod^{l}_{j=1}s_j \right)$.

  Now let us prove the other way round. Consider $m\in [1,\landing(n)]$ then
  $g^{\prime}_m(x)\in \Gsp$. It is therefore a product of conjugates of
  commutators of the form $[a_m,b_m]$, where $a_m\in A_m$ and $b_m\in B_m$.
  Applying \cref{Ex:Commutateur} with $x$ instead of $k_1$ we can show that we
  can write $[a_m,b_m]$ at $g_m(x)$ without changing any other entry in $\mbfg$
  (see also \cref{fig:Commutateur}). In a similar way, we can write a conjugate
  of $[a_m,b_m]$ at $g_m(x)$ without changing any other entry in $\mbfg$.
  Finally writing $(a_0,b_0)$ at the entry $g_0(x)$ writes $a_m$ at $g_m(0)$ and
  $b_m$ at $g_m(k_m)$ (see also \cref{fig:ab}). Therefore using
  \cref{Lmm:deffm} we can obtain 
  $(\mbfg,0)$ by first considering the word in $S_\BZ$ that writes
  all the values of $g_0$, then multiplying it on the left by a word that writes
  the value of $g^{\prime}_1$, and continue this process to write all
  $g^{\prime}_m$ for $m\leq \landing(n)$.

  Let us now check that the cursor
  remains in $[0,n]$ when writing $g_0$ and $g^{\prime}_m$.
  Take $m\in [1,\landing(n)]$, then $k_m\leq n$ and $\supp(g^\prime_m)$ is
  contained in $[k_m,n]$. Now let $x \in \supp(g^\prime_m)\subseteq [k_m,n]$.
  Since $\Gsp \subseteq \Gamma_m$, which is generated by $A_m \times
  B_m$, we can decompose $g^\prime_m(x)$ as a product of elements in $A_m$ and
  $B_m$. To write some $a_m\in A_m$ at the position $x$ the cursor needs to
  visit sites in $[0,x]$. To write some $b_m\in B_m$ it needs to visit sites in
  $[0,x-k_m]$. Therefore, the cursor remains in $[0,n]$ when writing $g_m(x)$ at
  position $x$. 
  Finally, for all $x$
  the cursor needs only to visit position $x$ in order to write $g_0(x)$. Since
  $\supp(g_0)$ is contained in $[0,n]$ then the cursor needs only to visit sites
  between $0$ and $n$.

  Combining what precedes with \cref{Lmm:deffm} and the hypothesis that $t\in
  [0,n]$, we get that the cursor needs only
  to visit cites between $[0,n]$ to write $(\mbfg,t)$.
  Hence the lemma.
\end{proof}

\subsubsection{Følner sequence}\label{Subsec:FolnerSequence}
In this subsection we describe a Følner sequence $(F_n)_{n\in \bN}$ for $\BZ$.
Recall that $\landing(n)$ denotes the integer such that $k_{\landing(n)}\leq n <
k_{\landing(n)+1}$.

\begin{Prop} \label{Prop:FolnerSequence}
  The following sequence is a Følner sequence of $\BZ$
  \begin{equation*}
    F_n:= \left\{ \left(\mbff, t \right) \mid 
       \range \left( \mbff,t\right) \subseteq \{0, \ldots, n-1\} \right\}.
  \end{equation*}
\end{Prop}

\begin{proof}
  Let $n\in \bN$ and $\delta:=(\mbff,t)\in F_n$. Remark that since
  $\delta$ belongs to $ F_n$, \cref{Lmm:range} implies that $t$ belongs to $\{0,\ldots,n-1\}$. Now let
  $s_1,\ldots,s_l \in S_{\BZ}$ such
  that $\delta=s_1\cdots s_l$ and take $s_{l+1}\in S_{\BZ}$. If
  $s_{l+1}=\big((a_m\delta_0),0 \big)$ for some $a \in A$ or if
  $s_{l+1}=\big((b_m\delta_{k_m}),0 \big)$ for some $b\in B$ then
  since the cursor of $s_{l+1}$ equals $0$,
  \begin{equation*}
   \range (\delta s_{l+1}) =\left\{\pi_2\left(\prod^{i}_{j=1}s_j \right) \,
      | \, i=1,\ldots,l+1 \right\} = \range(\delta).
  \end{equation*}
  Thus $\delta s_{l+1}\in F_n$.
  Finally denote by $[x,y]$ the range of $\delta$. Using the same formula
  as above we get
  \begin{align*}
   \range (\delta\cdot (\mathrm{id},1))\subseteq [x,y+1] &\quad \text{if} \ t=y,\\
   \range (\delta\cdot (\mathrm{id},1))\subseteq [x,y] &\quad \text{if} \ t<y.
  \end{align*}
  Hence for all $t<n-1$ we have $\range(\delta\cdot (\mathrm{id},1)) \subseteq
  [0,n-1]$. Now if $t=n-1$ then the cursor of $\delta
  (\mathrm{id},1)$ visits the site $n$, thus $\range(\delta\cdot
  (\mathrm{id},1))$ is not included in $[0,n-1]$ and therefore
  $\delta(\mathrm{id},1)$ does not belong to $F_n$. 

  A similar argument shows that $\delta(0,-1)$ belongs to $F_n$ if and only if
  $t \neq 0$. 
  Hence $\partial F_n= \left\{\left(\mbff , t \right)\in F_n \mid
    t\in \{0,n\} \right\}$ and thus
  \begin{equation*}
    {\left\vert \partial F_n \right\vert}/{\left\vert F_n \right\vert}=
    {2}/{n} \underset{n \rightarrow \infty}{\longrightarrow} 0.
  \end{equation*}
\end{proof}
\subsection{From the isoperimetric profile to the
  group}\label{Subsec:FromIPtoBZ}
We saw how to define a diagonal product from two sequences $(k_m)_m$ and
$(l_m)_m$. In this section we recall the definition given in \cite[Appendice
B]{BZ} of a Brieussel-Zheng group from its isoperimetric profile. 
We conclude with some useful results concerning the metric of these groups.
 
\subsubsection{\texorpdfstring{Definition of $\Delta$}{Definition of the
    diagonal product}}
Recall that in the particular case of expanders (see \cref{subsec:Expanders}) a
Brieussel-Zheng group is uniquely determined by the sequences
$(k_m)_{m\in \bN}$ and $(l_m)_{m\in \bN}$ (where $l_m$ corresponds to the
diameter of $\Gs$).
Thus, starting from a prescribed function $\rho$, we will define sequences
$(k_m)_{m\in \bN}$ and $(l_m)_{m\in \bN}$ such that the corresponding $\BZ$ 
verifies $\profile_{\BZ} \simeq \rho \circ \log$. 
Let
\begin{equation*}
  \calC:=\left\{ \zeta :[1,+\infty) \rightarrow [1,+\infty) \ \left\vert
    \begin{matrix}
      \zeta \ \text{continuous}, \ \ \zeta(1)=1 \\
      \zeta \ \text{and} \ x\mapsto x/\zeta(x) \text{non-decreasing}
    \end{matrix} \right.
  \right\}.
\end{equation*}
Equivalently this is the set of functions $\zeta$ satisfying $\zeta(1)=1$ and
\begin{equation}
  \label{Rq:IneqTildeRho}
 \left( \forall x,c\geq 1 \right) \quad \zeta(x)\leq \zeta(cx) \leq c\zeta(x).  
\end{equation}
So let $\rho \in \calC$. 
Combining \cite[Proposition B.2 and
Theorem 4.6]{BZ} we can show the following result (remember that with our
convention the isoperimetric profile considered in \cite{BZ} corresponds to
$1/I_{\BZ}$).

\begin{Prop}\label{Prop:DefdeDelta}
  Let $\kappa, \lambda \geq 2$. For any $\rho \in \calC$ there exists a
  subsequence $(k_m)_{m\in \bN}$ of $(\kappa^n)_{n\in \bN}$ and a subsequence
  $(l_m)_{m\in \bN}$ of $(\lambda^n)_{n\in \bN}$ such that the group $\BZ$ defined in
  \cref{subsec:Expanders} verifies $\profile_{\BZ}(x) \simeq \rho \circ \log$.
\end{Prop}

\begin{Ex}[{\cite[Example 4.5]{BZ}}]\label{Ex:Alpha} Let $\alpha >0$. If
  $\rho(x):=x^{1/(1+\alpha)}$ then the diagonal 
  product $\BZ$ defined by $k_m=\kappa^m$ and $l_m=\kappa^{\alpha m}$ verifies
  $I_{\BZ}\simeq \rho \circ \log$. 
\end{Ex}
\subsubsection{Technical tools} 
We recall the intermediate functions
defined in \cite[Appendix B]{BZ} and some of their properties.

Let $\rho \in \calC$ and let $f$ such that $\rho(x)=x/f(x)$. The construction of
a group corresponding to the given isoperimetric profile $\rho \circ \log$ is
based on the approximation of $f$ by a piecewise linear function $\bar{f}$.
For the quantification of orbit equivalence, many of our computations will use
$\bar{f}$ and some of its properties. We recall below all the
needed results, beginning with the definition of $\bar{f}$.

\begin{Lmm}
  \label{Prop:B2} Let $\rho\in \calC$ and $f$ such that $\rho(x)=x/f(x)$.
  Let $(k_m)$ and $(l_m)$ given by \cref{Prop:DefdeDelta} above and $\BZ$ the
  corresponding diagonal product. The function
  $\bar{f}$ defined by 
  \begin{equation}\label{Def:Tilde}
    \bar{f}(x):= \begin{cases}
      l_m & \text{if} \ x\in [k_ml_m,k_{m+1}l_m],\\
      \frac{x}{k_{m+1}} & \text{if} \ x\in [k_{m+1}l_m,k_{m+1}l_{m+1}],
    \end{cases}
  \end{equation}
  verifies $\bar{f} \simeq f$. In particular the map $\rhoaff$ defined by
  $\rhoaff(x)= x/\bar{f}(x)$ verifies $\rhoaff\simeq \rho$. 
\end{Lmm}
\begin{Ex}
  If $\rho(x)=x$ then $f(x)=1$ leads to $l_m=1$ for all $m$ and
  $k_m=+\infty$ for all $m\geq 1$. In this case $\BZ=(A\times B)\wr \bZ$.  
\end{Ex}

Remark that both $\bar{f}$ and $\rhoaff$ belong to $\calC$. In particular
they verify \cref{Rq:IneqTildeRho}, which is only true when $c$ and
$x$ are greater than $1$. When $c<1$ we get the following inequality.

\begin{Claim}\label{Claim:IneqTildeRho}
 If $0< c^\prime <1$ and $x^\prime \geq 1/c^\prime$ then $c^\prime\rhoaff(x^\prime)
 \leq\rhoaff(c^\prime x^\prime )$. 
\end{Claim}
\begin{proof}
  If $0< c^\prime <1$ then $1/c^\prime >1$, thus we can apply \cref{Rq:IneqTildeRho}
  with $c=1/c^\prime$ and $x=c^\prime x$ to obtain
 \begin{equation*} 
  \rhoaff(x^\prime) =\rhoaff\left(\frac{1}{c^\prime} c^\prime x^\prime\right)
      =\rhoaff(cx) \leq c\rhoaff(x)
    = \frac{1}{c^\prime}\rhoaff(c^\prime x^\prime).
  \end{equation*}
\end{proof}

\subsubsection{Metric}

We recall here some useful material about the metric of $\BZ$ and refer to
\cite[Section 2.2]{BZ} for more details. 
First, let $(x)_+:=\max\{x,0\}$.
\begin{Def}\label{Def:Em}
 For $j\in \bZ$ and $m\in \bN$ let $I^m_{j}:=[jk_m/2,(j+1)k_m/2-1]$.\\
 Let $f_m$ : $\bZ \rightarrow
 \Gamma_m$. The \deffont{essential contribution} of $f_m$ is defined as
 \begin{equation*}
   E_m(f_m):= k_m \sum_{j: \range(f_m,t)\cap I^m_j \neq \emptyset}
   \max_{x\in I^m_j} \left( |f_m(x)|_{\Gamma_m}-1 \right)_{+}.
 \end{equation*}
\end{Def}

The following proposition sums up \cite[Lemma 2.13, Proposition 2.14]{BZ}.
\begin{Prop}\label{Prop:metric}
  For any $\delta=(\mbff,t) \in \BZ$ we have
  \begin{align*}
    \left\vert {(\mbff,t)}\right\vert_{\BZ}
    & \leq 500 \sum^{\landing\left( \range(\delta) \right)}_{m=0} |(f_m,t)|_{\BZ_m},\\
     |(f_m,t)|_{\BZ_m} &\leq 9 \big( \range(f_m,t) + E_m(f_m)  \big).
  \end{align*}
\end{Prop}
\section{Folner tiling shifts}\label{Sec:FOTS}
We start by recalling some material of \cite{DKLMT} about Følner tiling
shifts and then construct such a tiling for diagonal products.

\subsection{Følner tiling shifts}\label{Subsec:FOTS}
The tools we are going to use to build orbit equivalence are \emph{Følner tiling
  shifts}\footnote{Delabie et al. \cite{DKLMT} use the term “Folner tiling
  sequence”. We chose to call $(\Sigma_n)_n$ a tiling \emph{shift} in 
    order to avoid confusion with usual Følner sequences.}. These 
sequences lead to Følner sequences defined recursively: the term of rank $(n+1)$ is
composed of a finite number of translates of the $n$-th term of the sequence. 

\begin{Def}\label{Def:FOTS}\index{Følner tiling shift}
  Let $G$ be an amenable group and $(\Sigma_n)_{n\in \bN}$ be a sequence of finite subsets
  of $G$. Define by induction the sequence $(T_n)_{n\in \bN}$ by $T_0:=\Sigma_0$ and
  $T_{n+1}:=T_n\Sigma_{n+1}$. We say that $(\Sigma_n)_{n\in \bN}$ is a (left) \deffont{Følner
    tiling shift} if
  \begin{itemize}
  \item $(T_n)_{n\in \bN}$ is a left Følner sequence, \textit{viz.} $\lim_{n \rightarrow \infty} {|gT_n \backslash T_n|}/{|T_n|} = 0$ for all $g\in G$;
  \item $T_{n+1}= \sqcup_{\sigma \in \Sigma_{n+1}} \sigma T_n$.
  \end{itemize}
  We call $\Sigma_n$ the set of \deffont{shifts}\index{Shifts} and $(T_n)_{n\in
    \bN}$ the \deffont{tiles}\index{Tiles}.
\end{Def}

We can also consider \emph{right} Følner tiling shifts, that is to say sequences
$(\Sigma_n)_n$ such that $T_{n+1}:=\Sigma_{n+1}T_n$ defines a right Følner
sequence.

\begin{Def}
  Let $S$ be a generating part of $G$. We say that $(\Sigma_n)_{n\in \bN}$ is a
  $(R_n,\varepsilon_n)$-Folner tiling shift if for all $n$ we have
  \begin{equation*}
    \diam{T_n} \leq R_n, \qquad \left\vert sT_n\backslash T_n \right\vert
    \leq \varepsilon_n |T_n| \quad (\forall s\in S).
  \end{equation*}
\end{Def}

Delabie et al. obtained in \cite{DKLMT} the two following examples. 
\begin{Ex}
 If $G=\bZ$ the sequence defined by $\Sigma_{n+1}:= \left\{ 0,2^n
 \right\}$ is a $(2^n,2^{1-n})$-Følner tiling shift and the sequence $(T_n)$
 thus defined verifies $T_n=[0,2^{n}-1]$.
\end{Ex}

\begin{Ex}
  If $G=(\bZ/2\bZ) \wr \bZ$ then the sequence $(\Sigma_n)_{n\in \bN}$ defined by 
  \begin{equation*}
    \begin{cases}
      \Sigma_{0}&:=\left\{ (f,0)\in G \mid \supp(f)\subseteq \{0,1\} \right\},\\
      \Sigma_{n+1}&:= \left\{ (f,0)\in G \mid \supp(f)\subseteq [2^n,2^{n+1}-1] \right\}\\
      &\phantom{:=}\cup \big\{ (f,2^n)\in G \mid \supp(f)\subseteq [0,2^{n}-1] \big\},
    \end{cases}
  \end{equation*}
  is a right $(3\cdot 2^{n},2^{-n})$-Følner tiling shift. Moreover the
  tiling $(T_n)_{n\in \bN}$ thus defined verifies $T_n=\left\{(f,m)\in G \mid
  \supp(f)\subseteq [0,2^{n}-1], \ m \in [0,2^{n}-1] \right\}$. 
\end{Ex}

In \cite{DKLMT} the authors used Følner tiling shifts to \emph{build} an explicit
orbit equivalence coupling between two amenable groups and \emph{quantify} its
integrability. Indeed if $G$ admits a
Følner tiling shift $(\Sigma_n)_{n\in \bN}$ then we can define $X:=\prod_{n\in \bN}
\Sigma_n$ and endow it with an action of $G$. Up to measure zero, two elements of
$X$ will be in the same orbit under that action if and only if they differ by
a finite number of indices. The equivalence relation thus induced is called the
\deffont{cofinite equivalence relation}. Now if $G^\prime$ admits a Følner
tiling shift $(\Sigma^{\prime}_n)_{n\in \bN}$ verifying
$|\Sigma_n|=|\Sigma^{\prime}_n|$ for all integer $n$, then there exists a
natural bijection between $X$ and $X^{\prime}:=\prod_{n\in
  \bN}\Sigma^{\prime}_n$ which preserves the cofinite equivalence relation. That
is to say $G$ and $H$ are orbit equivalent.
Furthermore they showed that if we know
the diameter and the ratio of elements in the boundary of each tile, then we can deduce
the integrability of the coupling. This is what the following proposition sums up.

\begin{Th}[{\cite[Prop. 6.6]{DKLMT}}] \label{Prop:ConditionpourOE} Let $G$ and
  $G^\prime$ be two discrete amenable groups and let $(\Sigma_n)_n$ be an
  $(\varepsilon_n,R_n)$-Følner tiling shift for $G$ and
  $(\Sigma^\prime_n)_n$ be an $({\varepsilon^\prime}_n,{R^\prime}_n)$-Følner 
  tiling shift for~${G^\prime}$.

  If $|\Sigma_n|=|{\Sigma^\prime}_n|$, then the groups are orbit equivalent over
  $X=\prod_{n\in \bN} \Sigma_n$. Moreover if $\varphi:\bR_+ \rightarrow \bR_+$
  is a non-decreasing map such that the sequence $\big( \varphi(2
    {R^\prime}_n)\left( \varepsilon_{n-1} -\varepsilon_n \right) \big)_{n\in
    \bN}$ is summable, then the coupling from $G$ to $G^\prime$ is
  $(\varphi,L^0)$-integrable.
\end{Th}

Using this tiling technique and the above theorem, Delabie et al.
\cite{DKLMT} obtained the first point of \cref{Ex:CouplagesQt} and the two
following quantifications.

\begin{Ex}\label{Ex:CouplageZmZn}
  For all $n$ and $m$
  there exists an orbit equivalence coupling from $\bZ^m$ to $\bZ^n$ which is
  $(\varphi_{\varepsilon},\psi_{\epsilon})$-integrable for every $\varepsilon
  >0$ where  
   \begin{equation*}
     \varphi_{\varepsilon}(x)=\frac{x^{n/m}}{\log(x)^{1+\varepsilon}} \quad
     \psi_{\varepsilon}(x)=\frac{x^{m/n}}{\log(x)^{1+\varepsilon}}.
  \end{equation*}
\end{Ex}

Remark that in particular
for all $p<n/m$ and $q<m/n$ there exists a $(L^p,L^{q})$-orbit equivalence
coupling from $\bZ^m$ to $\bZ^n$.

\begin{Ex}\label{Ex:LamplighterZ}
 Let $m\geq 2$. There exists an orbit equivalence coupling from $\bZ$ to
 $\bZ/m\bZ\wr \bZ$ that is $(\exp, \varphi_{\varepsilon})$-integrable for all
 $\varepsilon>0$ where
 \begin{equation*}
   \varphi_{\varepsilon}(x)=\frac{\log(x)}{\log(\log(x))^{1+\varepsilon}}.
 \end{equation*}
\end{Ex}
Note that the above example corresponds to the case when $\rho(x)=x$ in our
\cref{Th:CouplingwithZ}.

\subsection{Følner tiling shifts of diagonal products}\label{subsec:FOTSDelta}
Let $(k_m)_m$ and $(l_m)_m$ be two sequences verifying the conditions of
\textbf{(H)} and consider $\BZ$ the associated diagonal product (see
\cref{Sec:DefdeBZ}). We define below a Følner tiling shift for $\BZ$. Our goal
is to obtain a tiling verifying $T_n=F_{\kappa^n}$. After defining the shifts sets
$\Sigma_n$ we prove that the sequence $(\Sigma_n)_{n\in \bN}$ is actually a Følner
tiling shift. Finally we make this last statement precise by
computing $(R_n)_{n\in \bN}$ and $(\varepsilon_n)_{n\in \bN}$ such that 
$(\Sigma_n)_{n\in \bN}$ is a
$(R_n,\varepsilon_n)$-Følner tiling shift (see \cref{Def:FOTS}). 

\subsubsection{Definition of the shifts}\label{subsec:DefDesShifts}
For any $n\in \bN$, let $\Landing(n)=\landing(\kappa^n-1)$, that is to say $\Landing(n)$ is the
integer such that $k_{\Landing(n)}\leq {\kappa}^{n}-1 <k_{\Landing(n)+1}$. For
example if $k_n:=\kappa^n$ for all $n\in \bN$, then $\Landing(n)=n-1$.

Before defining our sequence $(\Sigma_n)_{n\in \bN}$, let us show some
practical results on $\Landing$. First remark that
since $(k_n)_{n\in \bN}$ is a subsequence of $(\kappa^n)_{n\in \bN}$, it
verifies $k_n\geq \kappa^n$ for all $n\in \bN$. Thus $\Landing(n)\leq n$ and
\begin{equation*}
  k_{\Landing (n)} < \kappa^n \leq k_{\Landing(n)+1}.
\end{equation*}

\begin{Claim} \label{Claim:Landingnplusun}
  Let $n\geq 0$, then either $\Landing(n+1)=\Landing (n)$ or
  $\Landing(n+1)=\Landing (n)+1$. Moreover in this second case
  $k_{\Landing(n+1)}=\kappa^{n}$. 
\end{Claim}
\begin{proof} Recall that by definition $\Landing (m)=\max
  \left\{i\in \bN\, | \, k_i \leq \kappa^m-1 \right\}$ for all $m\in \bN$.

  Let
  $n\in \bN$, then $\Landing(n+1)\geq \Landing (n)$. Moreover if $k_{\Landing
    (n)+1}\geq \kappa^{n+1}$ then $\Landing(n+1)< \Landing(n)+1$. That is to say
  $\Landing(n+1)\leq \Landing(n)$ and thus $\Landing(n+1)=\Landing (n)$.

  On the contrary, if $k_{\Landing(n)+1}<\kappa^{n+1}$ then $\Landing(n+1)\geq
  \Landing (n)+1$. But, by definition of $\Landing (n)$ it verifies $k_{\Landing
    (n)+1}\geq \kappa^{n}$ and by 
  construction of $(k_m)_{m\in \bN}$ we also have $k_{\Landing(n)+2} \geq \kappa
  k_{\Landing (n)+1}$ thus $k_{\Landing(n)+2} \geq \kappa^{n+1}$. Hence
  $\Landing(n+1)<\Landing(n)+2 $ and the first assertion.

  Finally if $\Landing(n+1)=\Landing(n)+1$ then by definition of $\Landing$
  \begin{equation*}
    k_{\Landing(n)} < \kappa^n \leq k_{\Landing(n)+1}=k_{\Landing(n+1)} \leq \kappa^{n+1}-1.
  \end{equation*}
  But $(k_m)_{m\in \bN}$ is a subsequence of $\kappa^m$ thus the above
  inequality implies $k_{\Landing(n+1)}=\kappa^n$. 
\end{proof}

Now, let us define the shifts. First let $\Sigma_0:=F_0$, then if $n\geq 0$
we distinguish two cases depending on whether $\Landing(n+1)=\Landing (n)$ or
$\Landing(n+1)=\Landing (n)+1$ and in both cases 
we split the set of shifts $\Sigma_{n+1}$ in $\kappa$ parts.

If $\Landing(n+1)=\Landing (n)$, let for all $j \in \{0, \ldots,\kappa -1\}$
\begin{equation*}
  \Sigma^j_{n+1}
  :=\left\{  \left( \mbfg, j\kappa^n \right)\in \BZ \ \left\vert \
      {\begin{aligned}
          \supp \left(g_0\right) &\subseteq 
          \left[ 0, j \kappa^n-1 \right] \cup \left[ (j+1)\kappa^n, \kappa^{n+1}-1 \right],\\
         \forall m \in [1,\Landing (n)] & \\
          \supp \left(g^\prime_m\right)& \subseteq 
          \big[k_m, j \kappa^n + k_m -1\big] \cup \big[(j+1)\kappa^n, \kappa^{n+1}-1 \big],\\
          \forall m \notin [0,\Landing (n)]& \\
          \supp \left(g^\prime_{m}\right)
          & =  \emptyset.
        \end{aligned}}
    \right. \right\}.
\end{equation*}
Now if $\Landing(n+1)=\Landing (n)+1$, we add the condition that $g^\prime_{\Landing(n)+1}$ has support
contained in $[k_{\Landing(n+1)},\kappa^{n+1}-1]$, namely for all $j \in \{0, \ldots,\kappa -1\}$
\begin{equation*}
  \Sigma^j_{n+1}
  :=\left\{  \left( \mbfg, j\kappa^n \right)\in \BZ \ \left\vert \
    {\begin{aligned}
        \supp \left(g_0\right)& \subseteq
        \left[ 0, j \kappa^n-1 \right]
        \cup \left[ (j+1)\kappa^n, \kappa^{n+1}-1 \right]\\
        \forall m \in [1,\Landing (n)]&\\
        \supp \left(g^\prime_m\right)&\subseteq
        \big[k_m, j \kappa^n + k_m -1\big]
        \cup \big[(j+1)\kappa^n, \kappa^{n+1}-1 \big],\\
        \supp \left(g^\prime_{\Landing (n)+1}\right) 
        &\subseteq\left[k_{\Landing (n)+1}, \kappa^{n+1}-1 \right], \\
        \forall m \notin [0,\Landing (n+1)] & \\
       \ \supp \left(g^\prime_{m}\right)
        &= \emptyset.
      \end{aligned}}
                        \right. \right\}.
\end{equation*}
Finally, in both cases we define $\Sigma_{n+1}:= \cup^{\kappa -1}_{j=0} \Sigma^j_{n+1}$.

Let $(\mbfg,t)$ be an element of some $\Sigma^j_{n+1}$. We represent in
\cref{fig:Sigmaj} the supports and the sets where the maps $g_0, {g^\prime}_1,
\ldots, {g^\prime}_{\Landing (n+1)}$ take their 
values. The light-blue rectangle with dotted outline is in
$\Sigma^j_{n+1}$ if and only if $\Landing(n+1)=\Landing(n)+1$.
\begin{figure}[htbp]
  \centering
  \includegraphics[width=0.85\textwidth]{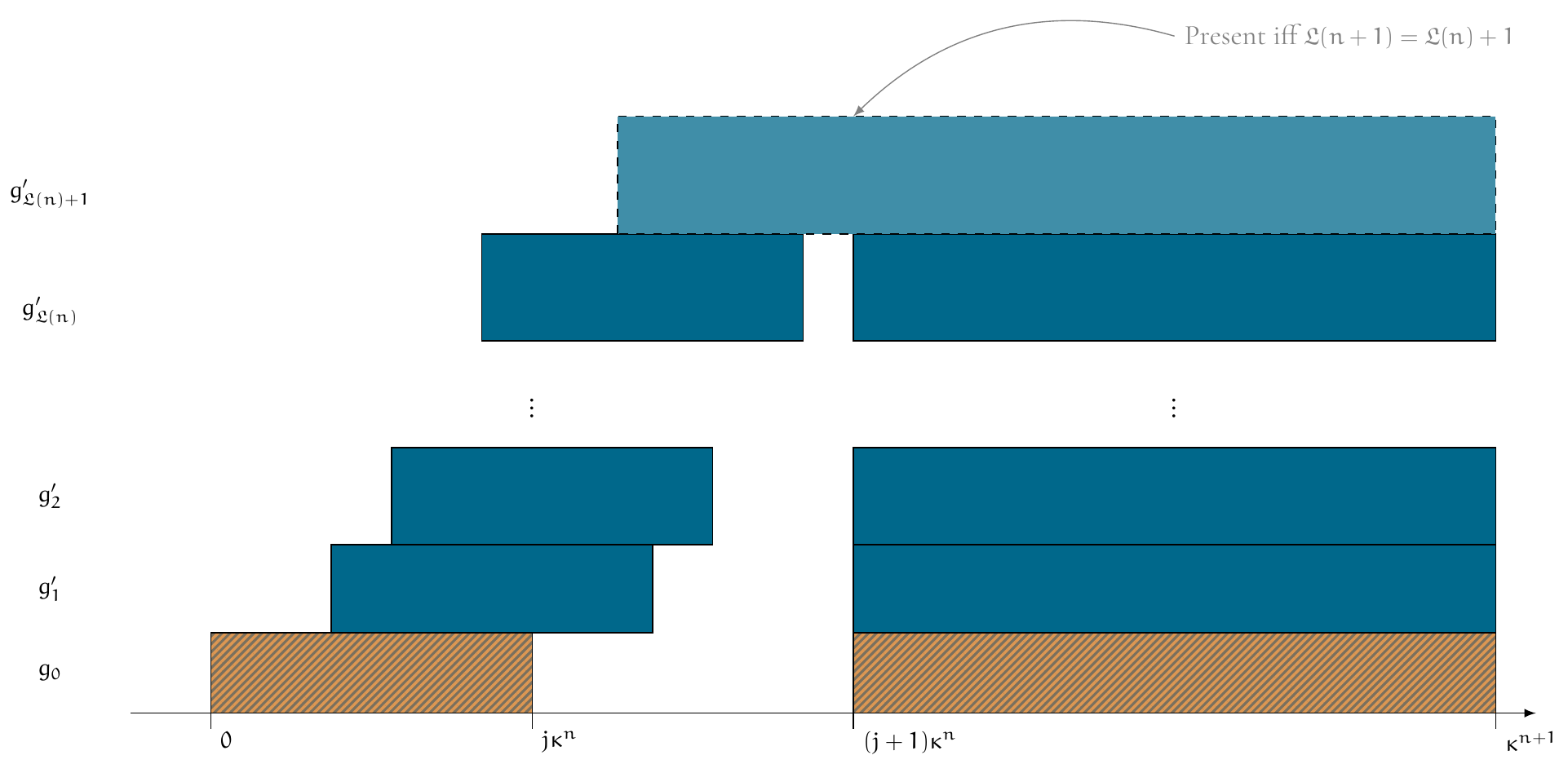}
  \caption{Support and values taken by $(\mbfg,t) \in \Sigma^j_{n}$}
  \label{fig:Sigmaj}
\end{figure}

Now that we have the shifts sequence, let us turn to the definition of the
tiles.
\subsubsection{Tiling}
Recall that $(F_n)_{n\in \bN}$ denotes the Følner sequence of $\BZ$ defined in
\cref{Prop:FolnerSequence}. The aim of this section is to show the theorem below. 
\begin{Th}\label{Th:FOTS}
  The sequence $(\Sigma_n)_{n\in\bN}$ defined in \cref{subsec:DefDesShifts} is a
  Følner tiling shift of $\BZ$. 
\end{Th}
Before showing that the sequence of tiles $(T_n)_{n\in \bN}$ thus induced
verifies indeed the conditions of \cref{Def:FOTS}, let us show the following lemma.

\begin{Lmm}\label{Lmm:DefTnBZ}
 The sequence $(T_n)_{n\in \bN}$ defined by $T_0:= F_0$ and $T_{n+1}:=
 \Sigma_{n+1} T_n$ for all $n>0$  verifies
 \begin{equation*} \big(\forall n \in \bN \big) \quad
   T_{n}=F_{\kappa^n}.
 \end{equation*}
\end{Lmm}
Let us discuss the idea of the proof. We proceed by induction and use a double
inclusion argument to prove the induction step. To
show that $\Sigma_{n+1}T_n$ is included in $F_{\kappa^{n+1}}$ we rely on
\cref{Lmm:range}, that is to say we verify that every element of
$\Sigma_{n+1}T_n$ has range included in $[0,\kappa^{n+1}-1]$. For the reversed
inclusion we consider an element $(\mbfh,t)$ of $F_{\kappa^{n+1}}$ and make the
elements $(\mbfg,j\kappa^n)$ of $\Sigma_{n+1}$ and $(\mbff,t^{\prime})$ of $T_n$ explicit 
such that $(\mbfh,t)=(\mbfg,j\kappa^n)(\mbff,t^{\prime})$. 

Mind the involved maps here: we study the values of $g_m$ and $f_m$ instead of the
“derived” functions $g^{\prime}_m$, $f^{\prime}_m$ usually considered.

\begin{proof}[Proof of the lemma]
  The assertion is true for $T_0$. Now let $n\geq 0$ and assume that
  $T_{n}=F_{\kappa^n}$. We show the induction step by double inclusion.
  \begin{center}
    \small \textsc{First inclusion} 
  \end{center}
  Let us prove that  $\Sigma_{n+1}T_{n} \subseteq F_{\kappa^{n+1}}$. Recall that
  $\Sigma_{n+1}=\cup^{\kappa-1}_{j=0} 
  \Sigma^j_{n+1}$.
  
  Let $\left( \mbff , t \right) \in T_n$ and $j \in \{0,
  \ldots, \kappa-1\}$. Take $ \left(\mbfg , j\kappa^{n} \right) \in
  \Sigma^j_{n+1} $, then the following product 
  \begin{equation*}
    \left(\mbfg ,  j\kappa^n\right)\left(\mbff, t \right)
    =  \Big( \big(g_mf_m\left(\cdot -j\kappa^n\right)\big)_m,
    t+ j\kappa^n \Big)
  \end{equation*}
  verifies $t + j\kappa^n \in
  \left[j \kappa^n,\kappa^{n}-1 + j\kappa^n \right]$ which is contained in $\left[
    0,\kappa^{n+1}-1 \right]$ since $j\leq \kappa-1$. Moreover 
  \begin{equation*}
    g_0(x)f_0(x-j\kappa^n)=
    \begin{cases}
      g_0(x)& \text{if} \ x \in  \left[ 0, j \kappa^n \right]
      \cup \left[ (j+1)\kappa^n, \kappa^{n+1}-1 \right]\\
      f_0(x-j\kappa^n)& \text{if} \ x \in \left[j\kappa^n,(j+1)\kappa^n-1 \right] \\
      0 & \text{else}.
    \end{cases}
  \end{equation*}
  Thus $\supp(g_0f_0(\cdot- j\kappa^n)) \subseteq \left[0,\kappa^{n+1}-1\right]$.
  Furthermore, for all $m \in \{1, \ldots,\Landing (n)\}$
  \begin{align*}
    \supp(g^\prime_m)
    &\subset \left[k_m, j \kappa^n + k_m -1\right]
      \cup \left[(j+1)\kappa^n, \kappa^{n+1}-1 \right]\\
    \supp\left(f^\prime_m\left(\cdot-j\kappa^n \right) \right)
    &\subseteq [j\kappa^n +k_m,(j+1)\kappa^{n}-1],
  \end{align*}
  hence by \cref{Claim:DeriveeProduit} the support of ${\left(
      g_mf_{m}(\cdot-j\kappa^m) \right)}^\prime$ is contained in $[k_m,\kappa^{n+1}-1]$. 

  Now if $\Landing(n+1)=\Landing (n)+1$ consider $m=\Landing
  (n)+1$. In that case $f^{\prime}_m\equiv \neutre$ since $m>\Landing(n)$.
  Thus ${\left(
      g_{m}f_{m}(\cdot-j\kappa^m) \right)}^\prime = g^\prime_n$ whose support is
  contained in $[k_{\Landing(n)+1}, \kappa^{n+1}-1]$.
 
   Finally ${\left( g_{m}f_{m}(\cdot-j\kappa^m)\right)}^\prime \equiv 0$ for all $m \notin [0,
   \Landing(n+1)]$.
   Hence by \cref{Lmm:range} the product $\left(\mbfg,j\kappa^n
   \right)(\mbff,t)$ has range included in $\left[0,\kappa^{n+1}-1\right]$ and thus belongs to $
   F_{\kappa^{n+1}}$. 

  \begin{center}
    \small \textsc{Second Inclusion} 
  \end{center}

  Let us show that $F_{\kappa^{n+1}}$ is contained in $\Sigma_{n+1}T_{n}$. So take
  $\left(\mbfh,t \right) $ in $ F_{\kappa^{n+1}}$. We want to define $\big(\mbff,
  t^\prime \big) \in T_{n}$ and $\left( 
    \mbfg, j\kappa^n \right) \in \Sigma_{n+1}$ such that
  $\left(\mbfg, j\kappa^n\right)\left(\mbff, t^\prime\right) =
  \left(\mbfh,t \right)$.
  First remark that $t<\kappa^{n+1}$ since $\left(\mbfh,t \right)$ belongs to
  $F_{\kappa^{n+1}}$. Thus there exists $t_0, \ldots, t_n$ in $[0,\kappa-1]$ 
  such that $t=\sum^n_{i=0}t_i \kappa^i$. Let $j=t_n$ and $t^\prime =
  \sum^{n-1}_{i=0}t_i \kappa^i$. Then $j$ does belong to $[0,\kappa-1]$ and
  $t^\prime$ to $[0,\kappa^n-1]$. We now have to define $\mbff$ and $\mbfg$ such
  that
  \begin{equation*}
    \left( {\left( g_m f_m\left( \cdot -j\kappa^n \right) \right)}_m
      , t^\prime +j\kappa^n  \right)
    = \big(\mbfh,t \big).
  \end{equation*}
  \begin{figure}[htbp]
    \centering
    \includegraphics[width=0.85\textwidth]{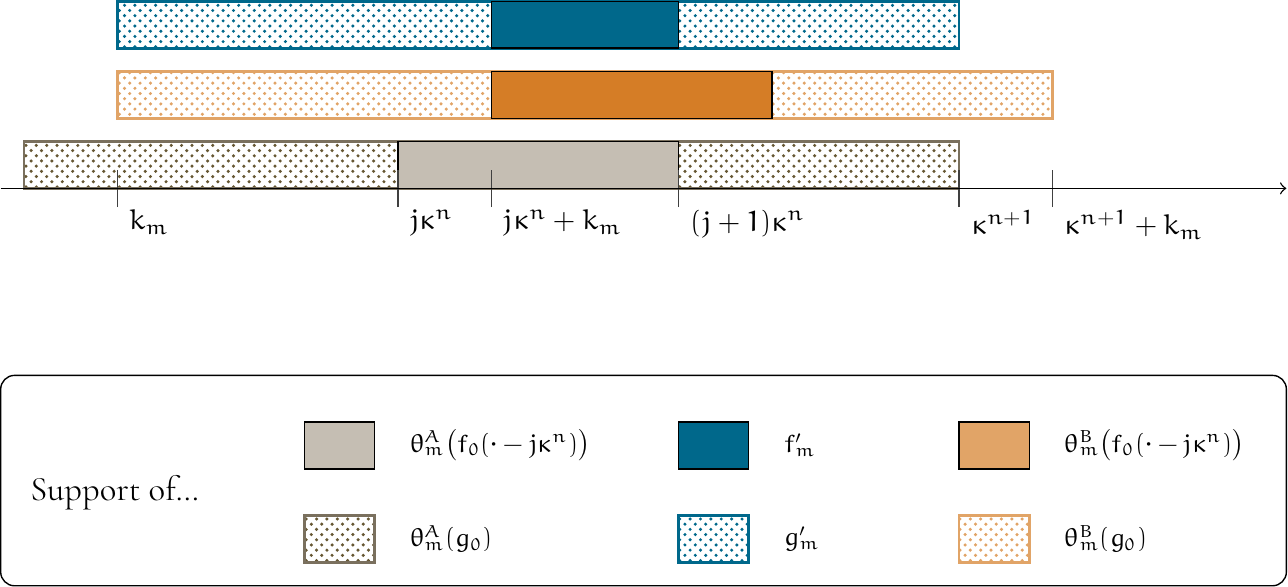}
    \caption{Supports}
    \label{fig:Supports1}
  \end{figure}
  We refer to \cref{fig:Supports1} for an illustration of the different
  supports. Let 
  \begin{align*}
    f_0(x)&:= \begin{cases} 
      h_0(x+j\kappa^n) &\text{if} \ x\in [0,\kappa^n-1],\\
      \neutre & \text{else,}
    \end{cases} \\
    g_0(x)&:= \begin{cases} 
      h_0(x) &\text{if} \ x\in [0,j\kappa^n-1] \cup [(j+1)\kappa^n, \kappa^{n+1}-1],\\
      \neutre & \text{else.}
    \end{cases}
  \end{align*}
  One can verify immediately that $ g_0 f_0\left( \cdot -j\kappa^n \right)=h_0$.
  Then take $m \in [1,\Landing (n)]$ and let 
  \begin{align*}
    {f^\prime}_m(x)
    &:= {\begin{cases}
        {h^\prime}_m(x+j\kappa^n) &\text{if} \ x\in [k_m,\kappa^n-1],\\
        \neutre & \text{else,}
      \end{cases} }\\
    {g^\prime}_m(x)
    &:= {\begin{cases}
        {h^\prime}_m(x) & \text{if}\ x \in [k_m,j\kappa^n+k_m-1]
        \cup [(j+1)\kappa^n,\kappa^{n+1}-1]\\ 
        \neutre &\text{else.}
      \end{cases}} 
  \end{align*}
  
  Now if $\Landing(n+1)=\Landing(n)+1$ then $k_{\Landing(n+1)}\geq
  \kappa^n$ and in that case define
  $g^{\prime}_{\Landing(n+1)}=h^{\prime}_{\Landing(n+1)}$. 
  Finally let ${f^\prime}_{\Landing(n+1)}\equiv \neutre$ and if
  $m> \Landing(n+1)$ let $ {g^\prime}_m\equiv \neutre \equiv {f^\prime}_m$.
  
  With the above definitions $\mbff$ and $\mbfg$ are uniquely defined. Moreover,
  by definition $(\mbfg,j\kappa ^n)$ belongs to
  $\Sigma^j_{n+1}$ and by \cref{Lmm:range} we have $\range(\mbff,t) \subseteq
  [0,\kappa^n-1]$ thus $(\mbff,t^\prime)$ belongs to $T_{n}$.
  Now, using \cref{Lmm:deffm} we verify that $g_mf_m(\cdot
  -j\kappa^n)=h_m$ thus $(\mbfh,t) \in \Sigma_{n+1}T_{n}$.
  \smallskip

  Hence, combining the first and second inclusion we get $F_{\kappa^{n+1}}=T_{n}$.
\end{proof}
We now know that $(T_n)_{n\in \bN}$ is a Følner sequence. To prove \cref{Th:FOTS} we have
to show that $(\Sigma_n)_{n\in \bN}$ a Følner \emph{tiling} shift.

\begin{proof}[Proof of {\cref{Th:FOTS}}]
  The sequence $(T_n)_{n\in \bN}$ is a Følner sequence, by the last lemma. Thus we
  only have to show that for all $\sigma \neq \tilde{\sigma} \in \Sigma_{n+1}$,
  $\sigma T_{n}  \cap \tilde{\sigma} T_{n}=\emptyset$. So let us denote by
  $(\mbfh,t)$ an element of $\sigma T_{n}  \cap \tilde{\sigma} T_{n}$. We
  distinguish two cases. 
  
  First if $\sigma \in \Sigma^j_{n+1}$ and $\tilde{\sigma} \in
  \Sigma^i_{n+1}$ for some $i \neq j$, then the cursor of $\sigma$ is equal to
  $j\kappa^n$ and the one of $\tilde{\sigma}$ to $i\kappa^n$. Thus
  \begin{align*}
    \left(\mbfh, t\right) \in \sigma T_{n} & \Rightarrow t \in [j \kappa^n, (j+1)\kappa^n-1],\\
    \left(\mbfh, t\right) \in \tilde{\sigma} T_{n} & \Rightarrow t \in [i\kappa^n,(i+1)\kappa^{n}-1].
  \end{align*}
  But since $i \neq j$ these two intervals are disjoint, thus $\sigma T_{n} \cap
  \tilde{\sigma} T_{n}=\emptyset$.
  \smallskip
   
  Now fix $j\in \{0, \ldots, \kappa-1\}$ and take $\sigma, \tilde{\sigma} \in
  \Sigma^{j}_{n+1}$. Let $\sigma:=\left(\mbfg ,j \kappa^n\right)$ and
  $\tilde{\sigma}:=\left(\tilde{\mbfg} ,j\kappa^n\right)$. Assume 
  that there exists $\left( \mbff,t \right), \left( \tilde{\mbff}, \tilde{t}
  \right) \in T_{n}$ such that $(\mbfg,j\kappa^n) \left( \mbff,t \right)= (\tilde{\mbfg},j\kappa^n)
  \left( \tilde{\mbff}, \tilde{t}\right)$. Then 
  \begin{equation} \label{eq:fgtiling}
    \forall m \in \bN,\ \forall x \in \bZ
    \quad g_mf_m(x- j\kappa^n)=\tilde{g}_m(x)\tilde{f}_m(x - j\kappa^n). 
  \end{equation}
  First remark that
  \begin{align*}
    \sigma, \tilde{\sigma} \in \Sigma^j_{n+1}
    \quad  &\Longrightarrow  \quad \supp(g_0), \, \supp(\tilde{g}_0)
             \subseteq [0,j\kappa^n-1] \cup [(j+1)\kappa^n, \kappa^{n+1}-1]\\
    \left( \mbff,t \right), \left( \tilde{\mbff}, \tilde{t}\right) \in T_{n}
    \quad &\Longrightarrow \quad  \supp({f_0(\cdot - j\kappa^n)}),\,
            \supp(\tilde{f}_0(\cdot - j\kappa^n)) \subseteq [j\kappa^n, (j+1)\kappa^n-1].
  \end{align*}
  In other word the support of $g_0$ (resp $\tilde{g}_0$) is disjoint from
  the one of $f_0(\cdot - j\kappa^n)$ (resp ${\tilde{f}}_0(\cdot - j\kappa^n)$). 
  Combining this with \cref{eq:fgtiling} we obtain that $g_0=\tilde{g}_0$ and
  $f_0=\tilde{f}_0$. 

  Now let $m>0$ and let us show that $g_m=\tilde{g}_m$. Due to supports overlap
  (see \cref{fig:Chevauchement}) we need to decompose
  $[0,\kappa^{n+1}-1]$ in five subintervals, namely
  \begin{align*}
    \Big[ 0, \kappa^{n+1}-1 \Big] =
    &\Big[ 0, j\kappa^n-1 \Big]\sqcup
      \Big[ j\kappa^n, j\kappa^n + k_m -1 \Big] \sqcup
      \Big[ j\kappa^n + k_m, (j+1)\kappa^n -1 \Big],\\
    &\sqcup
    \Big[(j+1)\kappa^n,(j+1)\kappa^n + k_m -1 \Big]\sqcup
    \Big[(j+1)\kappa^n + k_m,\kappa^{n+1} -1 \Big].
  \end{align*}
  
  If $x\leq j\kappa^n-1$ or $x\geq (j+1)\kappa^n+k_m$, then
  $f_m(x-j\kappa^n)=\neutre = \tilde{f}_m(x-j \kappa^n)$ 
  and thus $g_m(x)=\tilde{g}_m(x)$ by \cref{eq:fgtiling}.

  If $x\in [j\kappa^n, j\kappa^n + k_m -1]$ then using \cref{Lmm:deffm} and the fact
  that on that subinterval $f_0=\tilde{f}_0$, we get
  \begin{equation*}
    f_m(x-j\kappa^n) = \theta^A_0\left( f_0\left( x-j\kappa^n\right) \right)
    = \theta^A_0\left( {\tilde{f}}_0\left( x-j\kappa^n\right) \right)={\tilde{f}}_m(x-j\kappa^n). 
  \end{equation*}
  Hence by \cref{eq:fgtiling} we get $g_m(x)=\tilde{g}_m(x)$.
  
  If $x$ belongs to $ [j\kappa^n+k_m, (j+1) \kappa^n -1]$ then
  $g_m(x)=\tilde{g}_m(x)=\neutre$ and thus \cref{eq:fgtiling} implies that
  $f_m(x-j\kappa^n)=\tilde{f}_m(x-j\kappa^n)$, that is to say $f_m$ and 
  ${\tilde{f}}_m$ coincide on $[k_m, \kappa^n-1]$.
  
  Finally if $x\in [(j+1)\kappa^n,(j+1)\kappa^n+k_m-1]$ then using \cref{Lmm:deffm} and the fact
  that $f_0=\tilde{f}_0$ on that subinterval, we get
  \begin{equation*}
    f_m(x-j\kappa^n) = \theta^B_0\left( f_0\left( x-j\kappa^n-k_m \right) \right)
    = \theta^B_0\left( {\tilde{f}}_0\left( x-j\kappa^n-k_m \right) \right)={\tilde{f}}_m(x). 
  \end{equation*}
  Hence by \cref{eq:fgtiling}, we have $g_m(x)=\tilde{g}_m(x)$.

  Thus $\mbfg=\tilde{\mbfg}$ and then $\sigma=\tilde{\sigma}$. Which concludes
  the proof of the theorem.
\end{proof}
\begin{figure}[htbp]
  \centering
  \includegraphics[width=0.8\textwidth]{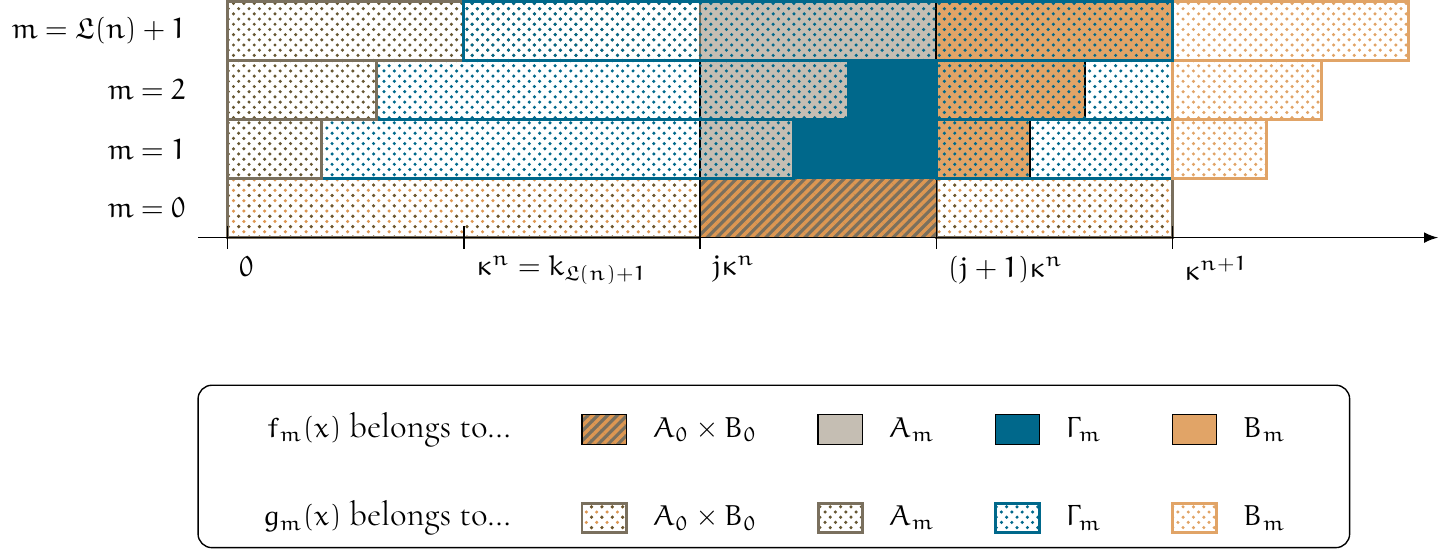}
  \caption{Supports overlap}
  \label{fig:Chevauchement}
\end{figure}

\subsubsection{Diameter and boundary}
Let us now quantify our shifts sequence.

\begin{Prop}\label{Lmm:RkEpskTn}
  The sequence $(\Sigma_n)_{n\in \bN}$ defined in \cref{subsec:DefDesShifts} is a
  $(R_n,\varepsilon_n)$-Følner tiling shift where
  \begin{equation*}
    R_n=\Cdiam \kappa^n l_{\Landing (n)} \qquad \varepsilon_n=\frac{2}{\kappa^{n}},
  \end{equation*}
  for some strictly positive constant $\Cdiam$.
\end{Prop}

First we prove the following lemma.
\begin{Lmm}\label{Lmm:DiamFn}
  There exists $\Cdiam>0$ depending only on $\BZ$ such that
  $\diam{F_n} \leq \Cdiam n l_{\landing(n-1)}$  for all $n\in \bN$.  
\end{Lmm}
To show this result, we use \cref{Prop:metric}.

\begin{proof} Let $n \in \bN$ and $(\mbff,t) \in F_n$. 
  First, take $m\leq \landing(n-1)$ and let us bound $E_m$ by above. Recall that
  $I^m_j=[jk_m/2,(j+1)k_m/2-1]$. Since $(f,t)$ belongs to $F_{n}$ its range is
  included in $[0,n-1]$, thus 
  \begin{align*}
    &\left\vert\left\{j \in \bZ: \range(f_m,t)\cap I^m_j \neq \emptyset\right\} \right\vert\\
    \leq&
      \left\vert \big\{j \in \bZ:\, [0,n-1]\cap [jk_m/2,(j+1)k_m/2-1]
      \neq \emptyset\big\} \right\vert,\\
    \leq& \left\vert \big\{j \in \bZ:\, jk_m/2\leq n -1\
      \text{and}\ (j+1)k_m/2\geq 1 \big\} \right\vert,\\
    \leq& \frac{2(n-2)}{k_m} + 1.
  \end{align*}
  Moreover remark that $|f_m(x)|_{\Gamma_m} \leq \diam{\Gamma_m} \leq \Clm l_m$
  for all $x$, thus 
  \begin{align*}
    E_m(f_m)
    &=  k_m \sum_{j: \range(f_m,t)\cap I^m_j \neq \emptyset}
      \max_{x\in I^m_j} \left( |f_m(x)|_{\Gamma_m}-1 \right)_{+},\\
    &\leq k_m  \sum_{j: \range(f_m,t)\cap I^m_j \neq \emptyset} l_m,\\
    &\leq k_m l_m \left(\frac{2(n-2)}{k_m}+1\right)=l_m(2(n-2)+k_m).
  \end{align*}
  Thus, applying the second part of \cref{Prop:metric} we get
  \begin{align*}
    |(f_m,t)|_{\BZ_m}\ &\leq\  9 \big( \range(f_m,t) + E_m(f_m)  \big) 
   \  \leq\  9\left(n +l_m\big(2(n-2)+k_m\big)\right).
  \end{align*}
  But if $m\leq\landing(n-1)$ then $k_m\leq n-1 \leq n$ thus we can bound $
  |(f_m,t)|_{\BZ_m}$ by above by $9n(3l_m+1)$. Now remark that
  $\landing\left(\range(\mbff,t) \right)\leq \landing(n-1)$. Thus, 
  using the preceding inequality and the first part of \cref{Prop:metric}, we
  get
  \begin{align*}
    \left\vert {(\mbff,t)}\right\vert_{\BZ}
     \leq 500 \sum^{\landing\left( \range(\mbff,t) \right)}_{m=0} |(f_m,t)|_{\BZ_m}
    & \leq 500 \sum^{\landing (n-1)}_{m=0} 9n\left(3l_m+1\right),\\
    & \leq 4500 n \sum^{\landing (n-1)}_{m=0}\left(3l_m+1  \right)
  \end{align*}
  Finally, since $l_m$ is a subsequence of a geometric sequence, there exists
  $C_l>0$ such that $ \sum^{\landing (n-1)}_{m=0}\left(3l_m+1  \right)\leq C_l
  l_{\landing (n-1)}$. Denoting $\Cdiam:=4500 C_l$ we get the lemma.
\end{proof}

Let us now show the wanted proposition.
\begin{proof}[Proof of \cref{Lmm:RkEpskTn}]
  First remark that by the proof of \cref{Prop:FolnerSequence} we have
  \begin{equation*}
    \varepsilon_n=\frac{|\partial T_n |}{|T_n|} =
    \frac{|\partial F_{\kappa^{n}} |}{|F_{\kappa^{n}}|}= \frac{2}{\kappa^n}.
  \end{equation*}
  Now by \cref{Lmm:DiamFn} we have $\diam{T_n}=\diam{F_{\kappa^n}}\leq \Cdiam
  \kappa^n l_{\Landing(n)}$. 
\end{proof}

\section{\texorpdfstring{Coupling with $\bZ$}{Coupling with the group of
    integers}}\label{Sec:CouplageZ}
Our aim in this section is to show \cref{Th:CouplingwithZ}. 
What we actually show is that a diagonal product $\BZ$ admits a coupling with
$\bZ$ satisfying \cref{Th:CouplingwithZ}. We start by defining a Følner tiling
shift for $\bZ$ in \cref{subsec:TilesZ}. We compute in \cref{Subsec:Estimates}
an estimate of the diameter of such tiles, namely the cardinal $|T_n|$. 
We conclude by showing the integrability of the coupling
using the criterion given by \cref{Prop:ConditionpourOE}. And then show that
$\BZ$ thus considered satisfies \cref{Th:CouplingwithZ}.
\subsection{\texorpdfstring{Tiles for $\bZ$}{Tiles for the integers}}\label{subsec:TilesZ}
We will denote by $(\SZ_n)_{n\in \bN}$ a Følner tiling shift of $\bZ$ and by
$(\TZ_n)_n$ the corresponding tiles. 

Consider $(\Sigma_n)_n$ and $(T_n)_n$ as defined in
\cref{subsec:DefDesShifts,Lmm:DefTnBZ} respectively. In order to use
\cref{Prop:ConditionpourOE} to get an orbit equivalence coupling between $\bZ$
and $\BZ$ we need $\Sigma_{n+1}$ and ${\SZ}_{n+1}$ to have the same number of
elements. We thus define
\begin{equation}\label{eq:TZn}
  \begin{cases}
    &\SZ_0=\big[0,|T_0|-1\big] \\
    \forall n \in \bN &
    \SZ_{n+1}:=\big\{ 0, |T_n|, 2|T_n|, \ldots,
    \left( \left\vert {\Sigma_{n+1}}\right\vert - 1 \right)|T_n|\big\}.
  \end{cases}
\end{equation}

It induces a sequence $({\TZ}_n)_{n\in \bN}$ defined by $\TZ_0=\SZ_0$ and
$\TZ_{n+1}=\SZ_{n+1}\TZ_n$ for all $n\geq 0$. We are going to prove that
$(\SZ_n)_{n\in \bN}$ is a Følner tiling shift for $\bZ$.

\begin{Prop}\label{Prop:FOTSZ}
  The sequence $(\SZ_n)_{n\in \bN}$ defined by \cref{eq:TZn} is a $(\RZ_n, \epsZ_n)$-Følner
  tiling shifts for $\bZ$ with
  \begin{equation*}
    \RZ_n= \left\vert{T_n}\right\vert \qquad \epsZ_n={2}/{\vert T_n \vert}.
  \end{equation*}
  Moreover the induced sequence $(\TZ_n)_{n\in \bN}$ verifies $\TZ_n=\left[ 0,
    \left\vert{T_n}\right\vert -1 \right]$ for all $n \in \bN$. 
\end{Prop}
\begin{proof} Let $(\SZ_n)_{n\in \bN}$ be as defined by \cref{eq:TZn} and recall that the
  induced tiling $(\TZ_n)_{n\in \bN}$ is the sequence defined by $\TZ_0:=\SZ_0$ and 
  $\TZ_{n+1}=\SZ_{n+1} \TZ_n$ for all $n \in \bN$.
  One can easily prove that for all $n\geq 0$
  \begin{equation}
    \TZ_n= \left[ 0, \left\vert{T_n}\right\vert -1 \right]. 
  \end{equation}
  It is now immediate to check that $\diam{\TZ_n}=|T_n|$ and $|\partial \TZ_n
  |/|\TZ_n|=2/|T_n|$. Furthermore note that if $\sigma, \sigma^\prime \in
  \SZ_{n+1}$ such that $\sigma\neq \sigma^{\prime}$ then
  $d_{\bZ}(\sigma,\sigma^\prime)\geq |T_n|=\diam{\TZ_n}$. Thus for such $\sigma$
  and 
  $\sigma^\prime$ we get $\sigma \TZ_n \cap \sigma^\prime T_n = \emptyset$.
  Therefore $(\Sigma_n)_{n\in \bN}$ is a Følner tiling shift and the proposition follows
  from the above quantifications on $T_n$.
\end{proof}

\subsection{Estimates: diameter and boundary}\label{Subsec:Estimates}
The integrability of the coupling between $\bZ$ and $\BZ$ depends on
$(R_n,\varepsilon_n)$ and $(\RZ_n,\epsZ_n)$ but by the above proposition, that last
couple depends on the value of the cardinality of the tiles $(T_n)_{n\in \bN}$.
The aim of this section is to give 
estimates of $|T_n|$ involving only terms of $(k_m)_{m\in \bN}$ and
$(l_m)_{m\in \bN}$.
First let us make the value of $|T_n|$ precise.
\begin{Lmm}\label{Lmm:CardTn}
  The sequence $(T_n)_n$ defined in \cref{Th:FOTS} verifies 
  \begin{equation*}
    |T_n|=\kappa^n \big( {|A| |B|} \big)^{\kappa^n}
    \prod^{\Landing(n)}_{m=1} {\left\vert {\Gamma^\prime}_m \right\vert}^{\kappa^{n}-k_m}.
  \end{equation*}
\end{Lmm}
\begin{proof}
  Recall that $T_n=F_{\kappa^n}=\left\{(\mbff, t) \mid \range
    \left(\mbff,t\right) \subseteq \{0, \ldots, \kappa^n-1\} \right\}$ for all
  $n\in \bN$. We use here \cref{Lmm:range} linking range and supports.
  Let $n\in \bN$ and take $(\mbff,t)\in T_n$, then there are exactly $\kappa^n$
  values of $t$ possible. Moreover $\mbff$ is uniquely determined by $f_0$ and
  $f^\prime_1, \ldots, f^\prime_{\Landing(n)}$ (see \cref{Lmm:deffm}). But $f_0$
  is supported on $[0,\kappa^n-1]$ which is set of cardinal $\kappa^n$ so  
  there are exactly $\big( {|A| |B|} \big)^{\kappa^n}$ possible values for
  $f_0$. Moreover if $m>0$ then remark that $f^\prime_m$ is supported on
  $[k_m,\kappa^n-1]$ which has
  $\kappa^n-k_m$ elements so there are exactly
  ${\left\vert \Gamma^\prime_m \right\vert}^{\kappa^n-k_m}$ possible values for
  $f^\prime_m$. Thus the number of elements in $T_n$ is
  \begin{equation*}
    {\kappa}^{n} \big( {|A| |B|} \big)^{\kappa^n}
    \prod^{\Landing(n)}_{m=1} {\left\vert \Gamma^\prime_m \right\vert}^{\kappa^n-k_m}.
  \end{equation*}
\end{proof}

Now let us bound $|T_n|$ so that the bounds depend only on $(\kappa^m)_{m\in \bN}$ and
$(l_m)_{m\in \bN}$.  

\begin{Prop}\label{Prop:EncadrementLogTn}
  There exists two constants $C_2, \CTn >0$ 
  such that for all $n \in \bN$,
  \begin{equation*}
    C_2 \kappa^{n-1}l_{\Landing(n)}\leq \ln |T_n| \leq \CTn \kappa^nl_{\Landing(n)}.
  \end{equation*}
\end{Prop}

Before showing the above proposition let us give an estimate of the right
factor of the expression of $|T_n|$.
 
\begin{Lmm}\label{Lmm:EncadrementLogGp}
  There exists two constants $C_1, C_2>0$ 
  such that for all $n \in \bN$,
  \begin{equation*}
    C_2 \kappa^{n-1} l_{\Landing(n)}
    \leq \ln \left(
      \prod^{\Landing(n)}_{m=1} {\left\vert {\Gamma^\prime}_m \right\vert}^{\kappa^n-k_m} \right)
    \leq C_1\kappa^n l_{\Landing(n)}.
  \end{equation*}
\end{Lmm}

\begin{proof}
  Recall that by \cref{eq:encadrementlngammap} there exists $c_1$, $c_2>0$ such that, for all $m$
  \begin{equation*}
    c_1 l_m -c_2 \leq \ln \left\vert {\Gamma_m}\right\vert \leq c_1l_m + c_2.
  \end{equation*}
  Since $\Gsp \leq \Gamma_m$ we thus have
  \begin{align*}
    \ln \left(
    \prod^{\Landing(n)}_{m=1} {\left\vert {\Gamma^\prime}_m \right\vert}^{\kappa^n-k_m} \right)
    &\leq \sum^{\Landing(n)}_{m=1} \left({\kappa^n-k_m}\right) \ln |\Gamma_m |,\\
    &\leq \sum^{\Landing(n)}_{m=1} \left({\kappa^n-k_m}\right)\left( c_1l_m+c_2 \right).
  \end{align*}
  But we can bound $\kappa^n-k_m$ from above by $\kappa^n$ and since
  $(l_m)_{m\in \bN}$ is a subsequence of a sequence having geometric growth, the sum 
  $\sum^{\Landing(n)}_{m=1} \left( c_1l_m+c_2 \right)$ is bounded from above by
  its last term up to a multiplicative constant. That is to say: there exists
  $C_1>0$ such that  
  \begin{equation*}
   \ln \left(
     \prod^{\Landing(n)}_{m=1} {\left\vert {\Gamma^\prime}_m \right\vert}^{\kappa^n-k_m} \right)
   \leq C_1\kappa^n l_{\Landing(n)}.
  \end{equation*}
  Hence the upper bound.
  Now, using that $[\Gamma_m:\Gsp]=|A||B|$ we have
  \begin{equation*}
    \ln \left(
    \prod^{\Landing(n)}_{m=1} {\left\vert {\Gamma^\prime}_m \right\vert}^{\kappa^n-k_m} \right)
    = \sum^{\Landing(n)}_{m=1} \left({\kappa^n-k_m}\right)
    \ln|\Gamma^\prime_m|
    = \sum^{\Landing(n)}_{m=1} \left({\kappa^n-k_m}\right)
    \ln \left(\frac{|\Gamma_{m} |}{|A||B|}\right). 
  \end{equation*}
  Bounding the sum from below by its last term and using once more
  \cref{eq:encadrementlngammap}, we get  
  \begin{align*}
    \ln \left(
    \prod^{\Landing(n)}_{m=1} {\left\vert {\Gamma^\prime}_m \right\vert}^{\kappa^n-k_m} \right)
    &\geq (\kappa^n-k_{\Landing(n)})\ln \left( \frac{|\Gamma_{\Landing(n)} |}{|A||B|} \right) ,\\
    &\geq (\kappa^n-k_{\Landing(n)})\left( c_1l_{\Landing(n)}-c_2 -\ln \left( |A||B| \right) \right),\\
    &\geq C_2(\kappa^n - k_{\Landing(n)})l_{\Landing(n)},
  \end{align*}
  for some $C_2>0$. We get the wanted inequality by noting that
  $\kappa^n - k_{\Landing(n)}\geq \kappa^{n-1}$.
\end{proof}
\begin{proof}[Proof of \cref{Prop:EncadrementLogTn}]
  Applying \cref{Lmm:EncadrementLogGp} to the cardinal of $T_n$ given by
  \cref{Lmm:CardTn} we obtain that there exists $\CTn>0$ such that $\ln|T_n|
  \leq \CTn  \kappa^n l_{\Landing(n)}$. Hence the upper bound.   
  The minoration comes imediately from \cref{Lmm:EncadrementLogGp}.
\end{proof}

Equipped with these bounds on $|T_n|$ we can now show the wanted integrability
for the coupling.

\subsection{Integrability of the coupling}
We will show that $\BZ$ is the group satisfying \cref{Th:CouplingwithZ}, but
first let us quantify the integrability of the orbit equivalence coupling with
$\bZ$ induced by the Følner tiling shifts we built.
Recall that $\calC$ denotes the set of non-decreasing functions $\rho\colon
[1,+\infty[ \rightarrow [1,+\infty[$ such that $x/\rho(x)$ is non-decreasing. 
\newpage

\begin{Th}\label{Th:OEZBZ}
  Let $\rho \in \calC$ and take $\BZ$ to be the Brieussel-Zheng diagonal
  product defined from $\rho$.
  Let $\varepsilon >0$ and $\Psi := \exp \circ \rho$ and let
  \begin{equation*}
    \varphi_{\varepsilon}(x):=\frac{\rho\circ \ln(x)}%
    {\left( \ln\circ \rho\circ \ln (x) \right)^{1+\varepsilon}}. 
  \end{equation*}
  There exists an orbit equivalence coupling from $\BZ$ to $\bZ$ that is
  $(\varphi_{\varepsilon},\Psi)$-integrable.
\end{Th}
Let us discuss the strategy of the proof. The demonstration is based on
\cref{Prop:ConditionpourOE}, thus we first prove that
$(\Psi(2R_n)\epsZ_{n-1})_n$ is summable and then that
$(\varphi_{\varepsilon}(2\RZ_n)\varepsilon_{n-1})_n$ 
is. In both cases we use \cref{Prop:EncadrementLogTn} to get upper bounds.
So far, we have the following quantifications.

\begin{HypBox}
  \centering
  \begin{tabular}{rlcrl}
     $R_n$&$=C_R \kappa^n l_{\Landing(n)}$ & \quad & $\RZ_n$&$= \left\vert{T_n}\right\vert$  \\
    $\varepsilon_n$&=${2}{\kappa^{-n}}$ & \quad & $\epsZ_n$&$={2}/{\left\vert T_n \right\vert}$ \\
  \end{tabular}
\end{HypBox}
\begin{proof}[Proof of \cref{Th:OEZBZ}]
  Let $\rho \in \calC$
  and take $\BZ$ to be the diagonal product defined from $\rho$ as described in
  \cref{Subsec:FromIPtoBZ}.  

  To begin, let us recall some preliminary results about $\rho$.
  Remember that $\rho \simeq \rhoaff$ where $\rhoaff$ is defined below
  \cref{Def:Tilde}. By definition of $\Landing(n)$ we have 
  $k_{\Landing(n)}l_{\Landing(n)} \leq \kappa^n l_{\Landing(n)} \leq
  k_{\Landing(n)+1}l_{\Landing(n)}$, thus by \cref{Def:Tilde}
  \begin{equation}\label{eq:tilderhokappan}
   \rhoaff(\kappa^n l_{\Landing(n)}) =\kappa^n. 
  \end{equation}
  
  Now let us show that the coupling from $\bZ$ to $\BZ$ is
  $\Psi$-integrable. To do so we prove that
  $\left(\Psi(2R_n)\epsZ_{n-1}  \right)$ is summable.
  First note that by \cref{Prop:EncadrementLogTn} we have the following lower bound on $|T_{n-1}|$
  \begin{equation}\label{eq:MinorationTnmoinsun}
    |T_{n-1}| \geq \exp\left(C_2 \kappa^{n-2} l_{\Landing(n-1)} \right).
  \end{equation}
  Moreover recall that $R_n=C_R \kappa^nl_{\Landing(n)}$ and $\epsZ_{n-1}=2/|T_{n-1}|$
  thus by the inequality above
  \begin{align*}
    \Psi(2R_n)\epsZ_{n-1}
    &=\exp\Big[ \rho(2C_R\kappa^nl_{\Landing(n)}) \Big] \frac{2}{|T_{n-1}|},\\
    &\leq 2 \exp 
      \Big[\rho \left(2C_R\kappa^nl_{\Landing(n)}\right)- C_2\kappa^{n-2}l_{\Landing(n-1)}\Big].
  \end{align*}
  But remember that $\rho\simeq \rhoaff$. Thus using
  \cref{eq:tilderhokappan,Rq:IneqTildeRho} we get 
  \begin{equation} \label{eq:jaipasdidee}
    \rho \left(2C_R\kappa^nl_{\Landing(n)}\right)
    \simeq \rhoaff \left(2C_R\kappa^nl_{\Landing(n)}\right)
    \leq 2C_R \rhoaff \left(\kappa^nl_{\Landing(n)}\right) = 2C_R \kappa^n.
  \end{equation}
  Combining the above result with the previous inequality, we get
  \begin{align*}
    \Psi(2R_n)\epsZ_{n-1}
      & \preccurlyeq 2 \exp\left[2C_R\kappa^n-C_2\kappa^{n-2}l_{\Landing(n-1)} \right],\\
      & = 2 \exp\left[\kappa^{n-2}\left(2C_R\kappa^2-C_2l_{\Landing(n-1)}\right)\right],
  \end{align*}
  which is
  summable. Indeed $l_{\Landing(n)}$ tends to infinity and thus
  $\left(2C_R\kappa^2-C_2l_{\Landing(n-1)}\right)<-1$ for $n$ large enough. Hence by
  \cref{Prop:ConditionpourOE} the orbit equivalence from $\bZ$ to $\BZ$ si $\Psi$-integrable.

  Now, let us show that for all $\varepsilon>0$ the coupling from $\BZ$ to $\bZ$ is
  $\varphi_{\varepsilon}$-integrable. Based on \cref{Prop:ConditionpourOE} we only
  have to prove that $\varphi_{\varepsilon}(2\RZ_n)\varepsilon_{n-1}$ is summable.
  Recall that $\RZ_n=|T_n|$ and $\varepsilon_{n-1}=2/\kappa^{n-2}$ and remark
  that by both the lower and upper bounds given in \cref{Prop:EncadrementLogTn}
  we have
  \begin{equation*}
    \varphi_{\varepsilon}(2\RZ_n)\varepsilon_{n-1}
    =\frac{2 \rho\circ\ln\big(2|T_n|\big)}%
    {\Big( \ln\circ \rho\circ \ln\big(2|T_n| \big)\Big)^{1+\varepsilon}\kappa^{n-1}}
    \leq \frac{2 \rho\big(2\CTn \kappa^nl_{\Landing(n)}\big)}%
    {\Big(\ln\circ\rho\left(2C_2 \kappa^{n-1}l_{\Landing(n)}\right)\Big)^{1+\varepsilon}\kappa^{n-1}}.
  \end{equation*}
  Let us give a lower bound for $\rho\left(2C_2
    \kappa^{n-1}l_{\Landing(n)}\right)$. Recall that $\rho \simeq \rhoaff$
  furthemore if $2C_2\geq 1$ then by \cref{eq:tilderhokappan} and since
  $\rhoaff$ is non-decreasing  
  \begin{align*}
    \kappa^{n-1}=\rhoaff\left(\kappa^{n-1}l_{\Landing(n)}\right)
    &\leq
    \rhoaff\left(2C_2 \kappa^{n-1}l_{\Landing(n)}\right)
    \simeq \rho\left(2C_2 \kappa^{n-1}l_{\Landing(n)}\right).   
  \end{align*}
  Now if $2C_2 <1$ using \cref{Claim:IneqTildeRho} with $c^\prime=2C_2$ and
  $x^\prime =\kappa^{n-1}l_{\Landing(n)}$ we get (for $n$ large enough)
  \begin{equation*}
    2C_2 \kappa^{n-1} =2C_2 \rhoaff(\kappa^{n-1}l_{\Landing(n)})
    \leq \rhoaff(2C_2\kappa^{n-1}l_{\Landing(n)})
   \simeq \rho\left( 2C_2 \kappa^{n-1}l_{\Landing(n)}\right)
  \end{equation*}
  Hence, in both cases $\kappa^{n-1} \preccurlyeq \rho( 2C_2
  \kappa^{n-1}l_{\Landing(n)})$. Finally replacing $C_R$ by $\CTn$ in 
  \cref{eq:jaipasdidee} we can show that $\rho
  \left(2\CTn \kappa^nl_{\Landing(n)}\right) \leq 2\CTn \kappa^n$. Thus,
  combining the two preceding results, we obtain
  \begin{align*}
    \varphi_{\varepsilon}(\RZ_n)\varepsilon_{n-1}
    &\leq \frac{2 \rho\big(\CTn \kappa^nl_{\Landing(n)}\big)}%
    {\Big( \ln\circ\rho\left(C_2 \kappa^{n-1}l_{\Landing(n)}\right)\Big)^{1+\varepsilon}\kappa^{n-1}}\\
    &\preccurlyeq\frac{\kappa^{n}}%
      {{\Big(\ln\left(\kappa^{n-1}\right)\Big)}^{1+\varepsilon}\kappa^{n-1}}
      =\frac{\kappa}{{\left((n-1)\ln(\kappa)\right)}^{1+\varepsilon}},
  \end{align*}
  which is a summable sequence. Hence by
  \cref{Prop:ConditionpourOE} the orbit equivalence coupling from $\BZ$ to $\bZ$
  si $\varphi_{\varepsilon}$-integrable.
\end{proof}
\begin{Rq}
 This result is stated in the general case, that is to say for an abstract
 $\rho$. Nonetheless, for some particular functions $\rho$ the quantification
 can be improved. For example the case where $k_n= 2^n$ and $l_n=2^{\alpha n}$
 corresponds to $\rho(x)\simeq x^{1/(1+\alpha)}$. In that case $\Landing(n)=n-1$ and we
 can show that the coupling from $\bZ$ to $\BZ$ is $\exp$-integrable (instead of
 $\exp\circ \rho$-integrable). Indeed, 
 let $c_{\varphi}< C_2/(C_R2^{3+\alpha})$ and $\Psi(x):=\exp(c_{\varphi} x)$,
 then by \cref{eq:MinorationTnmoinsun}
 \begin{align*}
    \Psi(2R_n)\epsZ_{n-1} &= \exp\left[c_{\varphi} 2C_Rk_nl_{n-1}\right] \frac{2}{|T_{n-1}|}\\
    &\leq \exp \left[c_{\varphi}2C_R2^n2^{\alpha(n-1)}- C_22^{n-2}2^{\alpha(n-2)}\right]
      2\\ 
   &= 2\exp \left[2^{n-2}2^{\alpha(n-2)} \left(c_{\varphi}C_R2^{3+\alpha}- C_2  \right)\right].
 \end{align*}
 Which is summable by choice of $c_{\varphi}$.
\end{Rq}
\begin{Rq} We can verify that the integrability obtained for the coupling from
  $\BZ$ to $\bZ$ is “almost” optimal.
  Indeed if the coupling from $\BZ$ to $\bZ$ is
  $\varphi$-integrable, then by \cref{Th:ProfiletOE} we have
  \begin{equation*}
    \varphi \circ I_{\bZ} \preccurlyeq I_{\BZ} 
  \end{equation*}
  where we recall that $I_{\bZ}(n) \simeq n$ and $I_{\BZ}(n)\simeq
  \rho \circ\ln(n)$. Thus using the inequality above, we get $\varphi(n)
  \preccurlyeq \rho \circ \ln(n)$. Hence the quantification of
  \cref{Th:OEZBZ} is optimal up to a logarithmic factor.
\end{Rq}

It is now easy to prove our first main theorem.
\begin{proof}[Proof of \cref{Th:CouplingwithZ}]
  Let $\rho \in \calC$
  and $\BZ$ to be the
  group defined in \cref{Prop:DefdeDelta}. By the aforementioned proposition it
  verifies $\profile_{\BZ}\simeq \rho \circ \log$. Moreover by
  \cref{Th:OEZBZ} there exists an orbit equivalence coupling from $\BZ$ and
  $\bZ$ that is $(\varphi_{\varepsilon},\exp \circ \rho)$-integrable for all
  $\varepsilon>0$.
\end{proof}

To prove \cref{Rq:Zd} we use the composition of couplings introduced in
\cite{DKLMT}.
We recall below the proposition concerning the integrability of this composition
and refer to \cite[Sections 2.3 and 2.5]{DKLMT} for more details on the
construction of the corresponding coupling.

\begin{Prop}[{\cite[Prop. 2.9 and 2.26]{DKLMT}}]\label{Prop:CompositionOE}
Let $\varphi,\psi:\bR^+\rightarrow \bR^+$ be non-decreasing subadditive maps with
$\varphi$ moreover concave.
If $(X_1,\mu_1)$ (resp. $(X_2,\mu_2)$) is a $(\varphi,L^0)$-integrable (resp.
$(\psi,L^0)$-integrable) orbit equivalence coupling from $\Gamma$ to $\Lambda$
(resp. $\Lambda$ to $\Sigma$), the composition of couplings gives a $(\varphi
\circ \psi,L^0)$-integrable orbit equivalence coupling from $\Gamma$
to~$\Sigma$.
\end{Prop}
Let us now show \cref{Rq:Zd} concerning the coupling with $\bZ^d$.
\begin{proof}[Proof of \cref{Rq:Zd}]\label{DemoCorollaire}
  Let $d \geq 1$. Let $\rho \in \calC$ and let $\BZ$ be the
  group defined in \cref{Prop:DefdeDelta}, in particular it verifies
  $\profile_{\BZ}\simeq \rho \circ \log$.
  Assume moreover that the map $\varphi_\varepsilon$ defined by $\varphi_{\varepsilon}(x):={\rho\circ
  \log(x)}/{\left( \log\circ\rho\circ\log(x) \right)^{1+\varepsilon}}$ is subadditive and concave.

  Since $d=1$ is precisely the case of \cref{Th:CouplingwithZ},
  we only have to treat the case of $d\geq 2$. For such a $d$ recall (see \cref{Ex:CouplageZmZn})
  that for all $p<d$ and
  all $q<1/d$ there exists a $(L^p,L^{q})$-integrable orbit equivalence coupling from
  $\bZ$ to $\bZ^d$. In particular taking $p=1$ and $q=0$ gives a $(L^1,L^{0})$-integrable orbit
  equivalence coupling from $\bZ$ to $\bZ^d$.
  Hence, using the composition of couplings described in \cite{DKLMT} 
  we can deduce from \cref{Th:CouplingwithZ} and \cref{Prop:CompositionOE} above
  that there exists a
  $(\varphi_{\varepsilon},L^0)$-integrable orbit equivalence coupling
  from $\BZ$ to $\bZ^d$. 
Hence the corollary.
\end{proof}
\begin{Rq}
We make the hypothesis that $\varphi_\varepsilon$ is subadditive and concave
only in order to use \cref{Prop:CompositionOE} and the composition of couplings. Building directly a coupling from $\BZ$ to $\bZ^d$ (instead of transiting \emph{via} $\bZ$)
might allow to remove the aforementioned assumption.
\end{Rq}

\newpage
\section{Conclusion and open problems}\label{Sec:Conclusion}
Let us conclude with some questions and remarks.

\subsection{Optimality and coupling building techniques}\label{subsec:Optim}

The tiling technique —though inspiring— is not always usable to get orbit
equivalence couplings. Indeed the condition that the two Følner tiling shifts
must have at each step the same cardinality is very restrictive. Furthemore this
technique does not seem to produce couplings with the best quantification: wether
it is our coupling with $\bZ$ or the one built in \cite{DKLMT}
(\cref{Ex:LamplighterZ,Ex:CouplageZmZn}) the integrability is always optimal
\emph{up to a logarithmic factor}.
One can thus ask: is the optimal integrability reachable? Is the logarithmic
error due to the building technique?

\subsection{Inverse problem}

We studied here the inverse problem for the group of integers
(\cref{Q:DKLMTZ}) but one can also ask the same question for other groups
than $\bZ$.

\begin{Q}\label{Q:InversePbGen}
  Given a function $\varphi$ and a group $H$ 
  is there a group $G$ such that there exists a $(\varphi,L^0)$-measure
  equivalent from $G$ to $H$? Can $G$ be chosen such that $\varphi\circ I_H
  \simeq I_G$?
\end{Q}

In \cite{EscSofic} we answer this question when $H$ is a diagonal
product, in particular $H$ can be a lamplighter group. This coupling is
obtained with another building technique than the tiling process and the
integrability is \emph{optimal}, answering the questions of
\cref{subsec:Optim} positively.

\newpage
\bibliographystyle{alpha}
\bibliography{Biblio}
\markboth{Bibliography}{Bibliography}
\newpage

\section*{Notations Index}
\markboth{Notations}{Notations}
\addcontentsline{toc}{section}{Notations}
\begin{description}
\item[$\preccurlyeq$, $\simeq$] See above \cref{Th:ProfiletOE}.
\item[$|X|$] Cardinal of the set $X$.
\item[$\partial F$] Boundary of the set $F$.
\item[$\BZ$] See \cref{Def:BZGr}.
\item[$\BZ_m$] See \cref{subsec:DefDelta}.
\item[$F_n$] Følner sequence of $\BZ$.
\item[$\mbfg$] The sequence of maps $(g_m)_{m\in \bN}$.
\item[$g^\prime_m$] See \cref{subsec:Derivedfunctions}.
\item[$\Gsp$] Normal closure of $[A_m,B_m]$.
\item[$\profile_{G}$] Isoperimetric profile of $G$.
\item[$R_n$] Diameter of $T_n$.
\item[$\RZ_n$] Diameter of $\TZ_n$.
\item[$\range(\mbff,t)$] The range of $(\mbff,t)$, see \cref{Def:Rangedelta}.
\item[$S_G$] A generating set of the group $G$.
\item[$\Sigma_n$] Følner tiling shifts (of $\BZ$).
\item[$\SZ_n$] Følner tiling shifts of $\bZ$.
\item[$T_n$] Tile of $\BZ$ defined by $T_n=\prod^{n}_{i=0}\Sigma_i$
\item[${\TZ}_n$] Tile of $\bZ$ defined by $\TZ_n=\prod^{n}_{i=0}\SZ_i$
\item[$\theta^A_m(f_m)$] Natural projection of $f_m$ on $A_m$ (see
  \cref{subsec:Derivedfunctions}).
\item[$\theta^B_m(f_m)$] Natural projection of $f_m$ on $B_m$ (see
  \cref{subsec:Derivedfunctions}).
\end{description}
\vfill

\noindent\textbf{Amandine Escalier}\\
Mathematisches Institut,\\
Fachbereich Mathematik und Informatik der Universität Münster,\\
Orléans-Ring 12,\\
48149 Münster,\\
Germany
\end{document}